\documentclass[10pt]{article}
\usepackage{amsmath}
\usepackage{amssymb}
\usepackage{amsthm}
\usepackage{mathrsfs}
\usepackage[compress,noadjust]{cite}
\numberwithin{equation}{section}

\begin{document}

\title{Estimates of some integral operators with bounded variable kernels on the Hardy and weak Hardy spaces over $\mathbb R^n$}
\author{Hua Wang \footnote{E-mail address: wanghua@pku.edu.cn.}\\
\footnotesize{College of Mathematics and Econometrics, Hunan University, Changsha 410082, P. R. China}}
\date{}
\maketitle

\begin{abstract}
In this paper, we first introduce $L^{\sigma_1}$-$(\log L)^{\sigma_2}$ conditions satisfied by the variable kernels $\Omega(x,z)$ for $0\leq\sigma_1\leq1$ and $\sigma_2\geq0$. Under these new smoothness conditions, we will prove the boundedness properties of singular integral operators $T_{\Omega}$, fractional integrals $T_{\Omega,\alpha}$ and parametric Marcinkiewicz integrals $\mu^{\rho}_{\Omega}$ with variable kernels on the Hardy spaces $H^p(\mathbb R^n)$ and weak Hardy spaces $WH^p(\mathbb R^n)$. Moreover, by using the interpolation arguments, we can get some corresponding results for the above integral operators with variable kernels on Hardy--Lorentz spaces $H^{p,q}(\mathbb R^n)$ for all $p<q<\infty$. 
\end{abstract}

\section{Introduction and main results}

Let $\mathbb R^n$($n\geq 2$) be the $n$-dimensional Euclidean space and $S^{n-1}$ be the unit sphere in $\mathbb R^n$ equipped with the normalized Lebesgue measure $d\sigma$. A function $\Omega(x,z)$ defined on $\mathbb R^n\times\mathbb R^n$ is said to belong to $L^{\infty}(\mathbb R^n)\times L^r(S^{n-1})$, $r\geq 1$, if it satisfies the following conditions:

$(i)$ for all $\lambda>0$ and $x,z\in\mathbb R^n$, $\Omega(x,\lambda z)=\Omega(x,z)$;

$(ii)$ $\big\|\Omega\big\|_{L^\infty(\mathbb R^n)\times L^r(S^{n-1})}:=\sup_{x\in\mathbb R^n}\left(\int_{S^{n-1}}|\Omega(x,z')|^r\,d\sigma(z')\right)^{1/r}<\infty$,

\noindent where $z'=z/{|z|}$ for any $z\in\mathbb R^n\backslash\{0\}$. Let $\Omega(x,z)\in L^{\infty}(\mathbb R^n)\times L^r(S^{n-1})$ satisfying the cancellation condition:
\begin{equation}\label{cancel}
\int_{S^{n-1}}\Omega(x,z')\,d\sigma(z')=0 \quad \mbox{for any } x\in\mathbb R^n.
\end{equation}
Then the singular integral operator with variable kernel is defined by
\begin{equation}
T_{\Omega}f(x)=\mbox{\upshape{P.V.}}\int_{\mathbb R^n}\frac{\Omega(x,x-y)}{|x-y|^n}f(y)\,dy.
\end{equation}

In 1955, Calder\'on and Zygmund \cite{cal1,cal2} investigated the $L^2$ boundedness of singular integral operators with variable kernels. They found that these operators $T_\Omega$ are closely related to the problem about second order elliptic partial differential equations with variable coefficients.  In \cite{cal1}, Calder\'on and Zygmund proved the following theorem (see also \cite{cal3}).

\newtheorem*{thma}{Theorem A}

\begin{thma}
Suppose that $\Omega(x,z)\in L^\infty(\mathbb R^n)\times L^r(S^{n-1})$ with $r>{2(n-1)}/n$, and satisfies $(\ref{cancel})$. Then there exists a constant $C>0$ independent of $f$ such that
\begin{equation*}
\big\|T_{\Omega}(f)\big\|_{L^2}\le C\big\|f\big\|_{L^2}.
\end{equation*}
\end{thma}

For $0<\alpha<n$ and $\Omega(x,z)\in L^{\infty}(\mathbb R^n)\times L^r(S^{n-1})$ with $r\geq 1$. Then the fractional integral operator with variable kernel is defined as follows:
\begin{equation}
T_{\Omega,\alpha}f(x)=\int_{\mathbb R^n}\frac{\Omega(x,x-y)}{|x-y|^{n-\alpha}}f(y)\,dy.
\end{equation}
In 1971, Muckenhoupt and Wheeden \cite{muckenhoupt} studied the $L^p$--$L^q$ boundedness of $T_{\Omega,\alpha}$ when $0<\alpha<n$, and obtained the following result (here, and in what follows we shall denote the conjugate exponent of $p>1$ by $p'=p/{(p-1)}$):

\newtheorem*{thmb}{Theorem B}

\begin{thmb}
Let $0<\alpha<n$, $1<p<n/{\alpha}$ and $1/q=1/p-\alpha/n$. Suppose that $\Omega(x,z)\in L^\infty(\mathbb R^n)\times L^r(S^{n-1})$ with $r>p'$, then there exists a constant $C>0$ independent of $f$ such that
\begin{equation*}
\big\|T_{\Omega,\alpha}(f)\big\|_{L^q}\le C\big\|f\big\|_{L^p}.
\end{equation*}
\end{thmb}

For $0<\rho<n$, in 1960, H\"ormander \cite{hor} defined the parametric Marcinkiewicz integral operator $\mu^{\rho}_{\Omega}$ of higher dimension as follows.
\begin{equation*}
\widetilde\mu^{\rho}_{\Omega}(f)(x)=
\left(\int_0^\infty\big|\widetilde F^{\rho}_{\Omega,t}(x)\big|^2\frac{dt}{t^{2\rho+1}}\right)^{1/2},
\end{equation*}
where
\begin{equation*}
\widetilde F^{\rho}_{\Omega,t}(x)=\int_{|x-y|\le t}\frac{\Omega(x-y)}{|x-y|^{n-\rho}}f(y)\,dy.
\end{equation*}
In this paper, we will consider the parametric Marcinkiewicz integral operator with variable kernel which is given by the following expression
\begin{equation}
\mu^{\rho}_{\Omega}(f)(x)=\left(\int_0^\infty\big|F^{\rho}_{\Omega,t}(x)\big|^2\frac{dt}{t^{2\rho+1}}\right)^{1/2},
\end{equation}
where
\begin{equation}
F^{\rho}_{\Omega,t}(x)=\int_{|x-y|\le t}\frac{\Omega(x,x-y)}{|x-y|^{n-\rho}}f(y)\,dy.
\end{equation}
When $\rho=1$, we shall denote $\mu^1_\Omega$ simply by $\mu_\Omega$, which was first defined and studied by Ding et al.in \cite{ding1,ding2} (for the convolution kernel case, see \cite{stein}).

Let $0\leq\sigma_1\leq1$ and $\sigma_2\geq0$. We say that $\Omega(x,z)$ satisfies the $L^{\sigma_1}$-$(\log L)^{\sigma_2}$ condition, if there exists an absolute constant $C>0$ such that
\begin{equation}\label{L-logL}
\sup_{x\in\mathbb {R}^n}\big|\Omega(x,y')-\Omega(x,z')\big|\leq C\big|y'-z'\big|^{\sigma_1}\cdot\frac{1}{\big(\log\frac{1}{|y'-z'|}\big)^{\sigma_2}}.
\end{equation}
holds uniformly in $y',z'\in S^{n-1}$. If $\sigma_1=0$ and $\sigma_2>0$, this new condition reduces to the logarithmic type Lipschitz condition, which was introduced and studied by Lee and Rim in \cite{lee}, when the kernel $\Omega$ does not depend on the first variable. If $0<\sigma_1\leq1$ and $\sigma_2=0$, this new condition is just the Lipschitz condition of order $\sigma_1$, which is actually stronger than $L^{\sigma_1}$-$(\log L)^{\sigma_2}$ condition assumed on the variable kernel $\Omega(x,z)$. In addition, it is obvious that if $\Omega(x,z)$ satisfies (\ref{L-logL}) for some $0\leq\sigma_1\leq1$ and $\sigma_2\geq0$, then $\Omega(x,z)$ is a bounded function in $\mathbb R^n\times\mathbb R^n$ and $\Omega(x,z)\in L^{\infty}(\mathbb R^n)\times L^r(S^{n-1})$ for any $1\leq r<\infty$.

The main purpose of this paper is to discuss the boundedness properties of $T_\Omega$, $T_{\Omega,\alpha}$ and $\mu^{\rho}_\Omega$ on the Hardy spaces $H^p(\mathbb R^n)$ and weak Hardy spaces $WH^p(\mathbb R^n)$, under the new $L^{\sigma_1}$-$(\log L)^{\sigma_2}$ condition (\ref{L-logL}) imposed on the variable kernel $\Omega(x,z)$. We now formulate our main results as follows.

\newtheorem{theorem}{Theorem}[section]

\begin{theorem}\label{mainthm:1}
Let $n\geq 2$, $n/{(n+1)}<p\leq1$, and $\Omega(x,z)$ satisfy $(\ref{cancel})$ and the $L^{\sigma_1}$-$(\log L)^{\sigma_2}$ condition $(\ref{L-logL})$. Then $T_{\Omega}$ is bounded from $H^p(\mathbb R^n)$ into $L^p(\mathbb R^n)$ provided that $\sigma_1$ and $\sigma_2$ satisfy either of the following\\
$(i)$ $\sigma_1=n(1/p-1)$ and $\sigma_2>1/p$;\\
$(ii)$ $n(1/p-1)<\sigma_1\leq1$ and $\sigma_2\geq0$.
\end{theorem}

\begin{theorem}\label{mainthm:2}
Let $0<\alpha<n$, $n/{(n+1)}<p\leq1$, $1/q=1/p-\alpha/n$ and $\Omega(x,z)$ satisfy the $L^{\sigma_1}$-$(\log L)^{\sigma_2}$ condition $(\ref{L-logL})$. Then $T_{\Omega,\alpha}$ is bounded from $H^p(\mathbb R^n)$ into $L^q(\mathbb R^n)$ provided that $\sigma_1$ and $\sigma_2$ satisfy either of the following\\
$(i)$ $\sigma_1=n(1/q-1)+\alpha$ and $\sigma_2>1/q$; \\
$(ii)$ $n(1/q-1)+\alpha<\sigma_1\leq1$ and $\sigma_2\geq0$.
\end{theorem}

\begin{theorem}\label{mainthm:3}
Let $1\leq \rho<n$, $n/{(n+1)}<p\leq1$, and $\Omega(x,z)$ satisfy $(\ref{cancel})$ and the $L^{\sigma_1}$-$(\log L)^{\sigma_2}$ condition $(\ref{L-logL})$. Then $\mu^{\rho}_{\Omega}$ is bounded from $H^p(\mathbb R^n)$ into $L^p(\mathbb R^n)$ provided that $\sigma_1$ and $\sigma_2$ satisfy either of the following\\
$(i)$ $\sigma_1=n(1/p-1)$ and $\sigma_2>1/p$; \\
$(ii)$ $n(1/p-1)<\sigma_1\leq1$ and $\sigma_2\geq0$.
\end{theorem}

\begin{theorem}\label{mainthm:4}
Let $n\geq 2$, $n/{(n+1)}<p\leq1$, and $\Omega(x,z)$ satisfy $(\ref{cancel})$ and the $L^{\sigma_1}$-$(\log L)^{\sigma_2}$ condition $(\ref{L-logL})$. Then $T_{\Omega}$ is bounded from $WH^p(\mathbb R^n)$ into $WL^p(\mathbb R^n)$ provided that $\sigma_1$ and $\sigma_2$ satisfy either of the following\\
$(i)$ $\sigma_1=n(1/p-1)$ and $\sigma_2>2/p$; \\
$(ii)$ $n(1/p-1)<\sigma_1\leq1$ and $\sigma_2\geq0$.
\end{theorem}

\begin{theorem}\label{mainthm:5}
Let $0<\alpha<n$, $n/{(n+1)}<p\leq1$, $1/q=1/p-\alpha/n$ and $\Omega(x,z)$ satisfy the $L^{\sigma_1}$-$(\log L)^{\sigma_2}$ condition $(\ref{L-logL})$. Then $T_{\Omega,\alpha}$ is bounded from $WH^p(\mathbb R^n)$ into $WL^q(\mathbb R^n)$ provided that $\sigma_1$ and $\sigma_2$ satisfy either of the following\\
$(i)$ $\sigma_1=n(1/q-1)+\alpha$ and $\sigma_2>1/q+\max\{1,1/q\}$;\\
$(ii)$ $n(1/q-1)+\alpha<\sigma_1\leq1$ and $\sigma_2\geq0$.
\end{theorem}

\begin{theorem}\label{mainthm:6}
Let $1\leq \rho<n$, $n/{(n+1)}<p\leq1$, and $\Omega(x,z)$ satisfy $(\ref{cancel})$ and the $L^{\sigma_1}$-$(\log L)^{\sigma_2}$ condition $(\ref{L-logL})$. Then $\mu^{\rho}_{\Omega}$ is bounded from $WH^p(\mathbb R^n)$ into $WL^p(\mathbb R^n)$ provided that $\sigma_1$ and $\sigma_2$ satisfy either of the following\\
$(i)$ $\sigma_1=n(1/p-1)$ and $\sigma_2>2/p$; \\
$(ii)$ $n(1/p-1)<\sigma_1\leq1$ and $\sigma_2\geq0$.
\end{theorem}

It should be pointed out that for the special case of $p=1$ and $\sigma_1=0$, Theorems \ref{mainthm:1}--\ref{mainthm:6} were already obtained by the author in \cite{wang}.

\section{Notations and preliminaries}

For any $0<p<\infty$, we denote by $L^p(\mathbb R^n)$ the classical Lebesgue spaces of all functions $f$ satisfying
\begin{equation}
\big\|f\big\|_{L^p}=\left(\int_{\mathbb R^n}|f(x)|^p\,dx\right)^{1/p}<\infty.
\end{equation}
When $p=\infty$, $L^\infty(\mathbb R^n)$ will be defined as follows:
\begin{equation}
\big\|f\big\|_{L^\infty}=\underset{x\in\mathbb R^n}{\mbox{ess\,sup}}\,|f(x)|<\infty.
\end{equation}
We also denote by $WL^p(\mathbb R^n)$ the weak $L^p$ spaces consisting of all measurable functions $f$ such that
\begin{equation}
\big\|f\big\|_{WL^p}=\sup_{\lambda>0}\lambda\cdot \big|\big\{x\in\mathbb R^n:|f(x)|>\lambda \big\}\big|^{1/p}<\infty.
\end{equation}

As we know, for any $0<p\leq1$, the Hardy spaces $H^p(\mathbb R^n)$ can be defined in terms of maximal functions. We write $\mathscr S(\mathbb R^n)$ to denote the Schwartz space of all rapidly decreasing smooth functions and $\mathscr S'(\mathbb R^n)$ to denote the space of all tempered distributions, i.e., the topological dual of $\mathscr S(\mathbb R^n)$.
Let $\varphi$ be a function in $\mathscr S(\mathbb R^n)$ satisfying $\int_{\mathbb R^n}\varphi(x)\,dx=1$.
Set
\begin{equation*}
\varphi_t(x)=t^{-n}\varphi(x/t),\quad t>0,\; x\in\mathbb R^n.
\end{equation*}
We will define the radial maximal function $M_\varphi(f)$ by
\begin{equation*}
M_\varphi f(x)=\sup_{t>0}\big|(\varphi_t*f)(x)\big|.
\end{equation*}
Then the Hardy spaces $H^p(\mathbb R^n)$ consist of those tempered distributions $f\in \mathscr S'(\mathbb R^n)$ for which
$M_\varphi(f)\in L^p(\mathbb R^n)$ with $\big\|f\big\|_{H^p}=\big\|M_\varphi(f)\big\|_{L^p}$. For $0<p\leq1$, one can characterize the Hardy spaces $H^p(\mathbb R^n)$ in terms of atoms in the following way (see \cite{coifman} and \cite{latter}).

\newtheorem{defn}[theorem]{Definition}

\begin{defn}
Let $0<p\leq1$, $1<q\leq\infty$, and the nonnegative integer $s\geq N=[n(1/p-1)]$, here $[x]$ indicates the integer part of $x$. A real-valued function $a(x)$ is said to be a $(p,q,s)$-atom centered at $x_0$ if the following conditions are satisfied:

$(a)$ $a\in L^q(\mathbb R^n)$ and is supported in a cube $Q$ centered at $x_0;$

$(b)$ $\|a\|_{L^q}\leq |Q|^{1/q-1/p};$

$(c)$ $\int_{\mathbb R^n}a(x)x^\gamma\,dx=0$ for every multi-index $\gamma$ with $|\gamma|\le s$.
\end{defn}

We will need the following atomic decomposition theorem for Hardy spaces $H^p(\mathbb R^n)$ given in \cite{coifman,latter}(For more details, the reader is referred to \cite{lu} and \cite{stein2}).

\begin{theorem} \label{thm:Hardy}
Let $0<p\leq1$ and $1<q\leq\infty$. For each $f\in H^p(\mathbb R^n)$, there exist a collection of $(p,q,[n(1/p-1)])$-atoms $\{a_j\}$ and a sequence of real numbers $\{\lambda_j\}$ with $\sum_j|\lambda_j|^p\leq C\|f\|^p_{H^p}$ such that $f=\sum_j\lambda_j a_j$ both in the sense of distributions and in the $H^p$ norm. Moreover,
\begin{equation*}
\big\|f\big\|_{H^p}\sim\bigg(\sum_j|\lambda_j|^p\bigg)^{1/p},
\end{equation*}
where the infimum is taken over all the above decompositions of $f\in H^p(\mathbb R^n)$ into $(p,q,[n(1/p-1)])$-atoms.
\end{theorem}

On the other hand, the weak $H^p$ spaces have first appeared in the work of Fefferman, Rivi\`ere and Sagher \cite{cfefferman}, which are the intermediate spaces between two Hardy spaces through the real method of interpolation. The atomic decomposition characterization of weak $H^1$ space on $\mathbb R^n$ was given by Fefferman and Soria in \cite{rfefferman}. Later, Liu \cite{liu} established the weak $H^p$ spaces on homogeneous groups for the whole range $0<p\le1$. The corresponding results related to $\mathbb R^n$ can be found in \cite{lu}. Let $0<p\le1$ and $N=[n(1/p-1)]$. Define
\begin{equation*}
\mathscr A_{N}=\Big\{\varphi\in\mathscr S(\mathbb R^n):\sup_{x\in\mathbb R^n}\sup_{|\alpha|\le N+1}(1+|x|)^{N+n+1}\big|D^\alpha\varphi(x)\big|\le1\Big\},
\end{equation*}
where $\alpha=(\alpha_1,\dots,\alpha_n)\in(\mathbb N\cup\{0\})^n$, $|\alpha|=\alpha_1+\dots+\alpha_n$, and
\begin{equation*}
D^\alpha\varphi=\frac{\partial^{|\alpha|}\varphi}{\partial x^{\alpha_1}_1\cdots\partial x^{\alpha_n}_n}.
\end{equation*}
For any given $f\in\mathscr S'(\mathbb R^n)$, the grand maximal function of $f$ is defined by
\begin{equation*}
G f(x)=\sup_{\varphi\in\mathscr A_{N}}\sup_{|y-x|<t}\big|(\varphi_t*f)(y)\big|.
\end{equation*}
Then we can define the weak Hardy spaces $WH^p(\mathbb R^n)$ by $WH^p(\mathbb R^n)=\big\{f\in\mathscr S'(\mathbb R^n):G(f)\in WL^p(\mathbb R^n)\big\}$. Moreover, we set $\big\|f\big\|_{WH^p}=\big\|G(f)\big\|_{WL^p}$. We also need the following atomic decomposition theorem for weak Hardy spaces $WH^p(\mathbb R^n)$ given in \cite{rfefferman,liu} (see also \cite{lu}).

\begin{theorem} \label{thm:weak Hardy}
Let $0<p\leq1$. For every $f\in WH^p(\mathbb R^n)$, then there exists a sequence of bounded measurable functions $\{f_k\}_{k=-\infty}^\infty$ with the following properties:

$(i)$ $f=\sum_{k=-\infty}^\infty f_k$ in the sense of distributions.

$(ii)$ Each $f_k$ can be further decomposed into $f_k=\sum_i b^k_i$, where $\{b^k_i\}$ satisfies

\quad $(a)$ Each $b^k_i$ is supported in a cube $Q^k_i$ with $\sum_{i}\big|Q^k_i\big|\le c2^{-kp}$, and $\sum_i\chi_{Q^k_i}(x)\le c$. Here $\chi_E$ denotes the characteristic function of the set $E$ and $c\sim\big\|f\big\|^p_{WH^p};$

\quad $(b)$ $\big\|b^k_i\big\|_{L^\infty}\le C2^k$, where $C>0$ is independent of $i$ and $k\,;$

\quad $(c)$ $\int_{\mathbb R^n}b^k_i(x)x^\gamma\,dx=0$ for every $i,k$ and every multi-index $\gamma$ with $|\gamma|\leq[n(1/p-1)]$.

Conversely, if $f\in\mathscr S'(\mathbb R^n)$ has a decomposition satisfying $(i)$ and $(ii)$, then $f\in WH^p(\mathbb R^n)$. Moreover, we have $\big\|f\big\|^p_{WH^p}\sim c.$
\end{theorem}

Throughout this paper, the letter $C$ always denote a positive constant independent of the main parameters involved, but it may be different from line to line. Moreover, we use $A\sim B$ to mean the equivalence of $A$ and $B$; that is, there exist two positive constants $C_1$ and $C_2$ independent of $A$, $B$ such that $C_1A\le B\le C_2A$.

\section{Boundedness on the Hardy spaces $H^p(\mathbb R^n)$}

\subsection{Proof of Theorem \ref{mainthm:1}}

\begin{proof}[Proof of Theorem $\ref{mainthm:1}$]
First we observe that for $n/{(n+1)}<p\leq1$, one has $[n(1/p-1)]=0$. Then in view of Theorem \ref{thm:Hardy} and Theorem A, it suffices to show that for any $(p,2,0)$-atom $a$, there exists a constant $C>0$ independent of $a$ such that $\big\|T_\Omega(a)\big\|_{L^p}\leq C$. Let $a(x)$ be a $(p,2,0)$-atom with supp\,$a\subseteq Q=Q(x_0,r_Q)$, and let $Q^*=2\sqrt nQ$, where $Q(x_0,r_Q)$ denotes the cube centered at $x_0$ with side length $r_Q$, all cubes are assumed to have their sides parallel to the coordinate axes and $\lambda Q$ denotes the cube concentric with $Q$ whose side length is $\lambda$ times as long. Then we write
\begin{equation*}
\begin{split}
\big\|T_\Omega(a)\big\|^p_{L^p}=\int_{\mathbb R^n}\big|T_\Omega(a)(x)\big|^p\,dx&=\int_{Q^*}\big|T_\Omega(a)(x)\big|^p\,dx
+\int_{(Q^*)^c}\big|T_\Omega(a)(x)\big|^p\,dx\\
&:=I_1+I_2.
\end{split}
\end{equation*}
Since the condition (\ref{L-logL}) implies that $\Omega(x,z)$ is bounded and $\Omega(x,z)\in L^{\infty}(\mathbb R^n)\times L^r(S^{n-1})$ for any $1<r<\infty$, then by using H\"older's inequality with exponent $\nu=2/p$, Theorem A and the size condition of atom $a$, we get
\begin{align}\label{thm1:I1}
I_1&\le\left(\int_{Q^*}\big|T_{\Omega}(a)(x)\big|^2\,dx\right)^{p/2}\left(\int_{Q^*}1\,dx\right)^{1-p/2}\notag\\
&\le C\cdot\big\|T_{\Omega}(a)\big\|^p_{L^2}|Q|^{1-p/2}\notag\\
&\le C\cdot\|a\|^p_{L^2}|Q|^{1-p/2}\notag\\
&\le C.
\end{align}
Let us now turn to estimate the other term $I_2$.By the vanishing moment condition of atom $a$, for any $x\in(Q^*)^c$, we have
\begin{equation*}
\begin{split}
\big|T_{\Omega}(a)(x)\big|&=
\left|\int_{Q}\bigg[\frac{\Omega(x,x-y)}{|x-y|^n}-\frac{\Omega(x,x-x_0)}{|x-x_0|^n}\bigg]a(y)\,dy\right|\\
&\le C\int_{Q}\left|\frac{1}{|x-y|^n}-\frac{1}{|x-x_0|^n}\right|\big|a(y)\big|\,dy\\
&\ +\int_{Q}\frac{|\Omega(x,x-y)-\Omega(x,x-x_0)|}{|x-x_0|^n}\big|a(y)\big|\,dy\\
&= \mbox{\upshape I+II}.
\end{split}
\end{equation*}
For the term I, notice that when $x\in(Q^*)^c$ and $y\in Q$, then we have $|x-y|\sim |x-x_0|$. Hence, we apply the mean value theorem and the size condition of atom $a$ to obtain
\begin{equation*}
\begin{split}
\mbox{\upshape I}&\le C\int_{Q}\frac{|y-x_0|}{|x-x_0|^{n+1}}\big|a(y)\big|\,dy\\
&\le C\cdot\frac{r_Q}{|x-x_0|^{n+1}}\int_{Q}\big|a(y)\big|\,dy\\
&\le C\cdot\frac{r_Q}{|x-x_0|^{n+1}}\cdot\|a\|_{L^2}|Q|^{1/2}\\
&\le C\cdot\frac{r_Q}{|x-x_0|^{n+1}}\cdot|Q|^{1-1/p}.
\end{split}
\end{equation*}
For the term II, in this case, we still have $|x-y|\sim |x-x_0|$ if $x\in(Q^*)^c$ and $y\in Q$. Thus
\begin{equation}\label{ineq-1}
\left|\frac{x-y}{|x-y|}-\frac{x-x_0}{|x-x_0|}\right|\le C\cdot\frac{r_Q}{|x-x_0|}.
\end{equation}
So by the condition (\ref{L-logL}) and the inequality (\ref{ineq-1}), we deduce that for any $x\in(Q^*)^c$,
\begin{align}\label{omega}
\Big|\Omega(x,x-y)-\Omega(x,x-x_0)\Big|&=\left|\Omega\Big(x,\frac{x-y}{|x-y|}\Big)
-\Omega\Big(x,\frac{x-x_0}{|x-x_0|}\Big)\right|\notag\\
&\leq\left(\frac{r_Q}{|x-x_0|}\right)^{\sigma_1}\cdot\frac{C}{\big(\log\frac{|x-x_0|}{r_Q}\big)^{\sigma_2}}.
\end{align}
Substituting the above inequality (\ref{omega}) into II and then using the size condition of atom $a$, we obtain
\begin{equation*}
\begin{split}
\mbox{\upshape II}&\le \left(\frac{r_Q}{|x-x_0|}\right)^{\sigma_1}\cdot\frac{C}{|x-x_0|^n\big(\log\frac{|x-x_0|}{r_Q}\big)^{\sigma_2}}
\int_{Q}\big|a(y)\big|\,dy\\
&\le \left(\frac{r_Q}{|x-x_0|}\right)^{\sigma_1}\cdot\frac{C}{|x-x_0|^n\big(\log\frac{|x-x_0|}{r_Q}\big)^{\sigma_2}}\cdot\|a\|_{L^2}|Q|^{1/2}\\
&\le C\cdot\frac{(r_Q)^{\sigma_1}}{|x-x_0|^{n+\sigma_1}\big(\log\frac{|x-x_0|}{r_Q}\big)^{\sigma_2}}\cdot|Q|^{1-1/p}.
\end{split}
\end{equation*}
Summarizing the above two estimates for I and II, we thus have
\begin{equation*}
\begin{split}
I_2&\le C\cdot (r_Q)^p|Q|^{p-1}\int_{(Q^*)^c}\frac{dx}{|x-x_0|^{p(n+1)}}\\
&+ C\cdot(r_Q)^{p\sigma_1}|Q|^{p-1}\int_{(Q^*)^c}\frac{dx}{|x-x_0|^{p(n+\sigma_1)}\big(\log\frac{|x-x_0|}{r_Q}\big)^{p\sigma_2}}\\
&=\mbox{\upshape III+IV}.
\end{split}
\end{equation*}
If we rewrite the above two integrals in polar coordinates, then we can get
\begin{align}\label{thm1:III}
\mbox{\upshape III}&\le C\cdot (r_Q)^{p+pn-n}\int_{|y|\geq(\sqrt{n})r_Q}\frac{dy}{|y|^{p(n+1)}}\notag\\
&\le C\cdot (r_Q)^{p+pn-n}\int_{(\sqrt{n})r_Q}^\infty\frac{s^{n-1}}{s^{p(n+1)}}ds\notag\\
&\le C,
\end{align}
where the last inequality holds since $p>n/{(n+1)}$, and
\begin{align*}
\mbox{\upshape IV}&\le C\cdot (r_Q)^{p\sigma_1+pn-n}\sum_{\ell=1}^\infty\int_{(\sqrt{n})^{\ell}r_Q\leq |y|<(\sqrt{n})^{\ell+1}r_Q}
\frac{dy}{|y|^{p(n+\sigma_1)}\big(\log\frac{|y|}{r_Q}\big)^{p\sigma_2}}\\
&\leq C\cdot (r_Q)^{p\sigma_1+pn-n}\sum_{\ell=1}^\infty\frac{1}{(\log\sqrt{n}\cdot\ell)^{p\sigma_2}}
\int_{(\sqrt{n})^{\ell}r_Q}^{(\sqrt{n})^{\ell+1}r_Q}\frac{s^{n-1}}{s^{p(n+\sigma_1)}}ds.
\end{align*}

Let us now consider the following two cases:

$(i)$ $p(n+\sigma_1)=n$ and $p\sigma_2>1$. Then we have
\begin{align}\label{thm1:IV1}
\mbox{\upshape IV}&\le C\sum_{\ell=1}^\infty\frac{\log\sqrt{n}}{(\log\sqrt{n}\cdot\ell)^{p\sigma_2}}\notag\\
&\le C.
\end{align}

$(ii)$ $p(n+\sigma_1)>n$ and $\sigma_2\geq0$. So we have
\begin{align}\label{thm1:IV2}
\mbox{\upshape IV}&\leq C\cdot (r_Q)^{p\sigma_1+pn-n}\sum_{\ell=1}^\infty\frac{1}{(\log\sqrt{n}\cdot\ell)^{p\sigma_2}}
\cdot\left[\frac{1}{(\sqrt{n})^{\ell}r_Q}\right]^{p(n+\sigma_1)-n}\int_{(\sqrt{n})^{\ell}r_Q}^{(\sqrt{n})^{\ell+1}r_Q}\frac{ds}{s}\notag\\
&\leq C\cdot (r_Q)^{p\sigma_1+pn-n}\sum_{\ell=1}^\infty\frac{\log\sqrt{n}}{(\log\sqrt{n}\cdot\ell)^{p\sigma_2}}
\cdot\left[\frac{1}{(\sqrt{n})^{\ell}r_Q}\right]^{p(n+\sigma_1)-n}\notag\\
&\le C\sum_{\ell=1}^\infty\left[\frac{1}{(\sqrt n)^{p(n+\sigma_1)-n}}\right]^{\ell}\notag\\
&\le C,
\end{align}
where the last series is convergent since $\big(\sqrt n\big)^{p(n+\sigma_1)-n}>1$. Combining the inequality (\ref{thm1:I1}) with (\ref{thm1:III})--(\ref{thm1:IV2}), we then complete the proof of Theorem \ref{mainthm:1}.
\end{proof}

\subsection{Proof of Theorem \ref{mainthm:2}}

\begin{proof}[Proof of Theorem $\ref{mainthm:2}$]
Observe that $[n(1/p-1)]=0$ by our assumptions. For $0<\alpha<n$, we first choose $p_1$ and $q_1$ in such a way that $1<p_1<n/{\alpha}$ and $1/{q_1}=1/{p_1}-\alpha/n$. Then by Theorem \ref{thm:Hardy} and Theorem B, it suffices to verify that for any $(p,p_1,0)$-atom $a$, there exists a constant $C>0$ independent of $a$ such that $\big\|T_{\Omega,\alpha}(a)\big\|_{L^q}\leq C$. Let $a(x)$ be a $(p,p_1,0)$-atom with supp\,$a\subseteq Q=Q(x_0,r_Q)$, and let $Q^*=2\sqrt nQ$. One writes
\begin{equation*}
\begin{split}
\big\|T_{\Omega,\alpha}(a)\big\|_{L^q}&=\left(\int_{\mathbb R^n}\big|T_{\Omega,\alpha}(a)(x)\big|^q\,dx\right)^{1/q}\\
&=\left(\int_{Q^*}\big|T_{\Omega,\alpha}(a)(x)\big|^q\,dx\right)^{1/q}
+\left(\int_{(Q^*)^c}\big|T_{\Omega,\alpha}(a)(x)\big|^q\,dx\right)^{1/q}\\
&:=J_1+J_2.
\end{split}
\end{equation*}
Since the condition (\ref{L-logL}) implies that $\Omega(x,z)$ is bounded and $\Omega(x,z)\in L^{\infty}(\mathbb R^n)\times L^r(S^{n-1})$ for all $r>p'_1>1$, notice also that $q_1>q$ and $1/p-1/q=1/{p_1}-1/{q_1}=\alpha/n$. Then by using H\"older's inequality with exponent $\nu={q_1}/q>1$, Theorem B and the size condition of atom $a$, we obtain
\begin{align}\label{thm2:J1}
J_1&\leq\left(\int_{Q^*}\big|T_{\Omega,\alpha}(a)(x)\big|^{q_1}\,dx\right)^{1/{q_1}}
\left(\int_{Q^*}1\,dx\right)^{(1-q/{q_1})\cdot1/q}\notag\\
&\leq C\cdot\big\|T_{\Omega,\alpha}(a)\big\|_{L^{q_1}}\big|Q\big|^{1/q-1/{q_1}}\notag\\
&\leq C\cdot\big\|a\big\|_{L^{p_1}}\big|Q\big|^{1/q-1/{q_1}}\notag\\
&\leq C\cdot\big|Q\big|^{1/{p_1}-1/p+1/q-1/{q_1}}\notag\\
&\leq C.
\end{align}
Now let us consider the other term $J_2$. By the vanishing moment condition of atom $a$, for any $x\in(Q^*)^c$, we have
\begin{equation*}
\begin{split}
\big|T_{\Omega,\alpha}(a)(x)\big|&=
\left|\int_{Q}\bigg[\frac{\Omega(x,x-y)}{|x-y|^{n-\alpha}}-\frac{\Omega(x,x-x_0)}{|x-x_0|^{n-\alpha}}\bigg]
a(y)\,dy\right|\\
&\le C\int_{Q}\left|\frac{1}{|x-y|^{n-\alpha}}-\frac{1}{|x-x_0|^{n-\alpha}}\right|
\big|a(y)\big|\,dy\\
&\ +\int_{Q}\frac{|\Omega(x,x-y)-\Omega(x,x-x_0)|}{|x-x_0|^{n-\alpha}}\big|a(y)\big|\,dy\\
&= \mbox{\upshape I+II}.
\end{split}
\end{equation*}
For the term I, note that when $x\in(Q^*)^c$ and $y\in Q$, then we have $|x-y|\sim |x-x_0|$. Applying the mean value theorem and the size condition of atom $a$, we get
\begin{equation*}
\begin{split}
\mbox{\upshape I}&\le C\int_{Q}\frac{|y-x_0|}{|x-x_0|^{n-\alpha+1}}\big|a(y)\big|\,dy\\
&\le C\cdot\frac{r_Q}{|x-x_0|^{n-\alpha+1}}\int_{Q}\big|a(y)\big|\,dy\\
&\le C\cdot\frac{r_Q}{|x-x_0|^{n-\alpha+1}}\cdot\big\|a\big\|_{L^{p_1}}\big|Q\big|^{1/{p'_1}}\\
&\le C\cdot\frac{r_Q}{|x-x_0|^{n-\alpha+1}}\cdot|Q|^{1-1/p}.
\end{split}
\end{equation*}
For the term II, as above we know that $|x-y|\sim |x-x_0|$ if $x\in(Q^*)^c$ and $y\in Q$. Thus, it follows from the previous inequalities (\ref{ineq-1}) and (\ref{omega}) that
\begin{equation*}
\begin{split}
\mbox{\upshape II}&\le \left(\frac{r_Q}{|x-x_0|}\right)^{\sigma_1}\cdot\frac{C}{|x-x_0|^{n-\alpha}\big(\log\frac{|x-x_0|}{r_Q}\big)^{\sigma_2}}
\int_{Q}\big|a(y)\big|\,dy\\
&\le \left(\frac{r_Q}{|x-x_0|}\right)^{\sigma_1}\cdot\frac{C}{|x-x_0|^{n-\alpha}\big(\log\frac{|x-x_0|}{r_Q}\big)^{\sigma_2}}
\cdot\big\|a\big\|_{L^{p_1}}\big|Q\big|^{1/{p'_1}}\\
&\le \left(\frac{r_Q}{|x-x_0|}\right)^{\sigma_1}\cdot\frac{C}{|x-x_0|^{n-\alpha}\big(\log\frac{|x-x_0|}{r_Q}\big)^{\sigma_2}}
\cdot|Q|^{1-1/p}.
\end{split}
\end{equation*}
Summing up the above two estimates for I and II, we thus obtain
\begin{equation*}
\begin{split}
J_2&\le C\cdot(r_Q)\cdot|Q|^{1-1/p}\left(\int_{(Q^*)^c}\frac{dx}{|x-x_0|^{q(n-\alpha+1)}}\right)^{1/q}\\
&+ C\cdot(r_Q)^{\sigma_1}\cdot|Q|^{1-1/p}\Bigg(\int_{(Q^*)^c}\frac{dx}{|x-x_0|^{q(n-\alpha+\sigma_1)}
\big(\log\frac{|x-x_0|}{r_Q}\big)^{q\sigma_2}}\Bigg)^{1/q}\\
&=\mbox{\upshape III+IV}.
\end{split}
\end{equation*}
Observe that $q>n/{(n-\alpha+1)}$ when $n/{(n+1)}<p\leq1$ and $1/q=1/p-\alpha/n$. We then use the polar coordinates for integrals to obtain
\begin{align}\label{thm2:III}
\mbox{\upshape III}&\leq C\cdot(r_Q)^{n+1-n/p}\left(\int_{|y|\geq(\sqrt{n})r_Q}\frac{dy}{|y|^{q(n-\alpha+1)}}\right)^{1/q}\notag\\
&\leq C\cdot(r_Q)^{n+1-n/p}\left(\int_{(\sqrt{n})r_Q}^\infty\frac{s^{n-1}}{s^{q(n-\alpha+1)}}ds\right)^{1/q}\notag\\
&\leq C.
\end{align}
and
\begin{align*}
\mbox{\upshape IV}&\le C\cdot(r_Q)^{n+\sigma_1-n/p}\Bigg(\sum_{\ell=1}^\infty\int_{(\sqrt{n})^{\ell}r_Q\leq |y|<(\sqrt{n})^{\ell+1}r_Q}
\frac{dy}{|y|^{q(n-\alpha+\sigma_1)}\big(\log\frac{|y|}{r_Q}\big)^{q\sigma_2}}\Bigg)^{1/q}\notag\\
&\le C\cdot(r_Q)^{n+\sigma_1-n/p}\Bigg(\sum_{\ell=1}^\infty\frac{1}{(\log\sqrt{n}\cdot\ell)^{q\sigma_2}}
\int_{(\sqrt{n})^{\ell}r_Q}^{(\sqrt{n})^{\ell+1}r_Q}\frac{s^{n-1}}{s^{q(n-\alpha+\sigma_1)}}ds\Bigg)^{1/q}.\notag\\
\end{align*}

We now consider the following two cases:

$(i)$ $q(n-\alpha+\sigma_1)=n$ and $q\sigma_2>1$. Then we have $n+\sigma_1=n/p$. Thus
\begin{align}\label{thm2:IV1}
\mbox{\upshape IV}&\le C\cdot(r_Q)^{n+\sigma_1-n/p}
\Bigg(\sum_{\ell=1}^\infty\frac{\log\sqrt{n}}{(\log\sqrt{n}\cdot\ell)^{q\sigma_2}}\Bigg)^{1/q}\notag\\
&\le C.
\end{align}

$(ii)$ $q(n-\alpha+\sigma_1)>n$ and $\sigma_2\geq0$.
\begin{align}\label{thm2:IV2}
\mbox{\upshape IV}&\le C\cdot (r_Q)^{n+\sigma_1-n/p}\Bigg(\sum_{\ell=1}^\infty\frac{\log\sqrt{n}}{(\log\sqrt{n}\cdot\ell)^{q\sigma_2}}
\cdot\left[\frac{1}{(\sqrt{n})^{\ell}r_Q}\right]^{q(n-\alpha+\sigma_1)-n}\Bigg)^{1/q}\notag\\
&\le C\cdot (r_Q)^{n+\sigma_1-n/p}\cdot\left(\frac{1}{r_Q}\right)^{n-\alpha+\sigma_1-n/q}
\Bigg(\sum_{\ell=1}^\infty\left[\frac{1}{(\sqrt n)^{q(n-\alpha+\sigma_1)-n}}\right]^{\ell}\Bigg)^{1/q}\notag\\
&\le C,
\end{align}
where in the last inequality we have used the facts that $\big(\sqrt n\big)^{q(n-\alpha+\sigma_1)-n}>1$ and $n/p=\alpha+n/q$. Therefore, by combining the inequality (\ref{thm2:J1}) with (\ref{thm2:III})--(\ref{thm2:IV2}), we then finish the proof of Theorem \ref{mainthm:2}.
\end{proof}

\subsection{Proof of Theorem \ref{mainthm:3}}

In \cite{wang}, we have already proved the following result, which will be used in the proof of our main theorem.

\begin{theorem}\label{Mar:L2}
Suppose that $1\leq \rho<n$, $\Omega(x,z)\in L^\infty(\mathbb R^n)\times L^r(S^{n-1})$ with $r>{2(n-1)}/n$, and satisfies $(\ref{cancel})$. Then there exists a constant $C>0$ independent of $f$ such that
\begin{equation*}
\big\|\mu^{\rho}_{\Omega}(f)\big\|_{L^2}\le C\big\|f\big\|_{L^2}.
\end{equation*}
\end{theorem}

\begin{proof}[Proof of Theorem $\ref{mainthm:3}$]
Again, we have $[n(1/p-1)]=0$ when $n/{(n+1)}<p\leq1$. Arguing as in the proof of Theorem \ref{mainthm:1}, by using Theorem \ref{thm:Hardy} and Theorem \ref{Mar:L2}, we only have to show that for any $(p,2,0)$-atom $a$, there exists a constant $C>0$ independent of $a$ such that $\big\|\mu^{\rho}_\Omega(a)\big\|_{L^p}\leq C$. Let $a(x)$ be a $(p,2,0)$-atom with supp\,$a\subseteq Q=Q(x_0,r_Q)$, and let $Q^*=2\sqrt nQ$. Then we have
\begin{equation*}
\begin{split}
\big\|\mu^{\rho}_\Omega(a)\big\|^p_{L^p}=\int_{\mathbb R^n}\big|\mu^{\rho}_\Omega(a)(x)\big|^p\,dx&=\int_{Q^*}\big|\mu^{\rho}_\Omega(a)(x)\big|^p\,dx
+\int_{(Q^*)^c}\big|\mu^{\rho}_\Omega(a)(x)\big|^p\,dx\\
&:=K_1+K_2.
\end{split}
\end{equation*}
Applying H\"older's inequality with exponent $\nu=2/p$, Theorem \ref{Mar:L2} and the size condition of atom $a$, we get
\begin{equation*}
\begin{split}
K_1&\le\left(\int_{Q^*}\big|\mu^{\rho}_{\Omega}(a)(x)\big|^2\,dx\right)^{p/2}
\left(\int_{Q^*}1\,dx\right)^{1-p/2}\notag\\
&\le C\cdot\big\|\mu^{\rho}_{\Omega}(a)\big\|^p_{L^2}|Q|^{1-p/2}\notag\\
&\le C\cdot\|a\|^p_{L^2}|Q|^{1-p/2}\notag\\
&\le C.
\end{split}
\end{equation*}
Let us now turn to deal with the term $K_2$. If for $1\leq\rho<n$, we set
$$\mathcal K^{\rho}(x,z)=\frac{\Omega(x,z)}{|z|^{n-\rho}}\chi_{\{|z|\le1\}}(z) \quad \mbox{and} \quad \mathcal K^{\rho}_t(x,z)=\frac{1}{t^n}\cdot \mathcal K^{\rho}\Big(x,\frac{z}{t}\Big).$$
Then
\begin{equation}
\mu^{\rho}_{\Omega}(a)(x)=\left(\int_0^\infty\bigg|\int_{\mathbb R^n}\mathcal K^{\rho}_t(x,x-y)a(y)\,dy\bigg|^2\frac{dt}{t}\right)^{1/2}.
\end{equation}
For any $x\in(Q^*)^c$, supp\,$\mathcal K^{\rho}(x,\cdot)\subseteq B(0,1)$, the unit ball in $\mathbb R^n$. Then using the vanishing moment condition of atom $a$, we have
\begin{equation*}
\begin{split}
\bigg|\int_{\mathbb R^n}\mathcal K^{\rho}_t(x,x-y)a(y)\,dy\bigg|=\ &\bigg|\int_{Q}\Big[\mathcal K^{\rho}_t(x,x-y)-\mathcal K^{\rho}_t(x,x-x_0)\Big]a(y)\,dy\bigg|\\
\le\ &\frac{C}{t^{\rho}}\cdot\int_{Q}
\bigg|\frac{1}{|x-y|^{n-\rho}}-\frac{1}{|x-x_0|^{n-\rho}}\bigg|\big|a(y)\big|\,dy\\
&+\frac{1}{t^{\rho}}\cdot\int_{Q}\frac{|\Omega(x,x-y)-\Omega(x,x-x_0)|}{|x-x_0|^{n-\rho}}
\big|a(y)\big|\,dy\\
=\ &\mbox{\upshape I+II}.
\end{split}
\end{equation*}
When $x\in(Q^*)^c$ and $y\in Q$, then $|x-y|\sim |x-x_0|$. Using the mean value theorem and the size condition of atom $a$, we obtain
\begin{equation*}
\begin{split}
\mbox{\upshape I}&\le \frac{C}{t^{\rho}}\cdot\int_{Q}\frac{|y-x_0|}{|x-x_0|^{n-\rho+1}}
\big|a(y)\big|\,dy\\
&\le C\cdot\frac{r_Q}{t^{\rho}|x-x_0|^{n-\rho+1}}\int_{Q}\big|a(y)\big|\,dy\\
&\le C\cdot\frac{r_Q}{t^{\rho}|x-x_0|^{n-\rho+1}}\cdot\|a\|_{L^2}|Q|^{1/2}\\
&\le C\cdot\frac{r_Q}{t^{\rho}|x-x_0|^{n-\rho+1}}\cdot|Q|^{1-1/p}.
\end{split}
\end{equation*}
On the other hand, from the previous inequalities (\ref{ineq-1}) and (\ref{omega}), it follows that
\begin{equation*}
\begin{split}
\mbox{\upshape II}&\le C\cdot\frac{(r_Q)^{\sigma_1}}{t^{\rho}|x-x_0|^{n-\rho+\sigma_1}\big(\log\frac{|x-x_0|}{r_Q}\big)^{\sigma_2}}
\int_{Q}\big|a(y)\big|\,dy\\
&\le C\cdot\frac{(r_Q)^{\sigma_1}}{t^{\rho}|x-x_0|^{n-\rho+\sigma_1}\big(\log\frac{|x-x_0|}{r_Q}\big)^{\sigma_2}}
\cdot\|a\|_{L^2}|Q|^{1/2}\\
&\le C\cdot\frac{(r_Q)^{\sigma_1}}{t^{\rho}|x-x_0|^{n-\rho+\sigma_1}\big(\log\frac{|x-x_0|}{r_Q}\big)^{\sigma_2}}\cdot|Q|^{1-1/p}.
\end{split}
\end{equation*}
Recall that for any fixed $x$, supp\,$\mathcal K^{\rho}(x,\cdot)\subseteq\{z\in\mathbb R^n:|z|\leq 1\}$, then for any $y\in Q$ and $x\in(Q^*)^c$, we have
\begin{equation}
t\geq|x-y|\geq|x-x_0|-|y-x_0|\geq\frac{|x-x_0|}{2}.
\end{equation}
Therefore
\begin{equation*}
\begin{split}
\big|\mu^{\rho}_\Omega(a)(x)\big|&\le C\cdot|Q|^{1-1/p}
\Bigg[\frac{r_Q}{|x-x_0|^{n-\rho+1}}+
\frac{(r_Q)^{\sigma_1}}{|x-x_0|^{n-\rho+\sigma_1}\big(\log\frac{|x-x_0|}{r_Q}\big)^{\sigma_2}}\Bigg]\\
&\times\left(\int_{\frac{|x-x_0|}{2}}^\infty\frac{dt}{t^{2\rho+1}}\right)^{1/2}\\
&\le C\cdot|Q|^{1-1/p}\Bigg[\frac{r_Q}{|x-x_0|^{n+1}}+
\frac{(r_Q)^{\sigma_1}}{|x-x_0|^{n+\sigma_1}\big(\log\frac{|x-x_0|}{r_Q}\big)^{\sigma_2}}\Bigg].
\end{split}
\end{equation*}
The rest of the proof is exactly the same as that of Theorem \ref{mainthm:1}, we can also obtain
\begin{equation*}
K_2\le C.
\end{equation*}
Summing up all the above estimates, we conclude the proof of Theorem \ref{mainthm:3}.
\end{proof}

\section{Boundedness on the weak Hardy spaces $WH^p(\mathbb R^n)$}

\subsection{Proof of Theorem \ref{mainthm:4}}

\begin{proof}[Proof of Theorem $\ref{mainthm:4}$]

For any given $\lambda>0$, we may choose $k_0\in\mathbb Z$ such that $2^{k_0}\le\lambda<2^{k_0+1}$. For every $f\in WH^p(\mathbb R^n)$, then by Theorem \ref{thm:weak Hardy}, we can write
\begin{equation*}
f=\sum_{k=-\infty}^\infty f_k=\sum_{k=-\infty}^{k_0} f_k+\sum_{k=k_0+1}^\infty f_k:=F_1+F_2,
\end{equation*}
where $F_1=\sum_{k=-\infty}^{k_0} f_k=\sum_{k=-\infty}^{k_0}\sum_i b^k_i$, $F_2=\sum_{k=k_0+1}^\infty f_k=\sum_{k=k_0+1}^\infty\sum_i b^k_i$ and $\{b^k_i\}$ satisfies $(a)$--$(c)$ in Theorem \ref{thm:weak Hardy}. Then we have \begin{equation*}
\begin{split}
&\lambda^p\cdot \big|\big\{x\in\mathbb R^n:\big|T_{\Omega}(f)(x)\big|>\lambda\big\}\big|\\
\le\ &\lambda^p\cdot \big|\big\{x\in\mathbb R^n:\big|T_{\Omega}(F_1)(x)\big|>\lambda/2\big\}\big|
+\lambda^p\cdot \big|\big\{x\in\mathbb R^n:\big|T_{\Omega}(F_2)(x)\big|>\lambda/2\big\}\big|\\
:=\ &I_1+I_2.
\end{split}
\end{equation*}
First we claim that the following inequality holds:
\begin{equation}\label{F1:L2}
\big\|F_1\big\|_{L^2}\le C\cdot\lambda^{1-p/2}\big\|f\big\|^{p/2}_{WH^p}.
\end{equation}
In fact, since supp\,$b^k_i\subseteq Q^k_i=Q\big(x^k_i,r^k_i\big)$ and $\big\|b^k_i\big\|_{L^\infty}\le C 2^k$ by Theorem \ref{thm:weak Hardy}, where $Q\big(x^k_i,r^k_i\big)$ denotes the cube centered at $x^k_i$ with side length $r^k_i$. Hence, it follows from Minkowski's inequality for integrals that
\begin{equation*}
\big\|F_1\big\|_{L^2}\leq\sum_{k=-\infty}^{k_0}\Big\|\sum_i b^k_i\Big\|_{L^2}.
\end{equation*}
For each $k\in\mathbb Z$, by using the bounded overlapping property of the cubes $\big\{Q^k_i\big\}$ and the fact that $1-p/2>0$, we thus obtain
\begin{equation*}
\begin{split}
\Big\|\sum_i b^k_i\Big\|_{L^2}&\leq\sup_i\big\|b^k_i\big\|_{L^\infty}\left(\int_{x\in\cup_i Q^k_i}1\,dx\right)^{1/2}\\
&\leq C\cdot2^k\bigg(\sum_i \big|Q^k_i\big|\bigg)^{1/2}\\
&\leq C\cdot2^{k(1-p/2)}\big\|f\big\|^{p/2}_{WH^p},
\end{split}
\end{equation*}
and
\begin{equation*}
\begin{split}
\big\|F_1\big\|_{L^2}&\le C\sum_{k=-\infty}^{k_0}2^{(k-k_0)(1-p/2)}\cdot\lambda^{1-p/2}\big\|f\big\|^{p/2}_{WH^p}\\
&= C\cdot\lambda^{1-p/2}\big\|f\big\|^{p/2}_{WH^p}\sum_{k=0}^\infty\left(\frac{\,1\,}{2}\right)^{k(1-p/2)}\\
&\le C\cdot\lambda^{1-p/2}\big\|f\big\|^{p/2}_{WH^p}.
\end{split}
\end{equation*}
By the hypothesis, we know that $T_{\Omega}$ is bounded on $L^2(\mathbb R^n)$ according to Theorem A. This fact together with Chebyshev's inequality and the inequality (\ref{F1:L2}) yields
\begin{equation*}
\begin{split}
I_1&\leq \lambda^p\cdot\frac{4}{\lambda^2}\big\|T_{\Omega}(F_1)\big\|^2_{L^2}\\
&\leq C\cdot\lambda^{p-2}\big\|F_1\big\|^2_{L^2}\\
&\leq C\big\|f\big\|^p_{WH^p}.
\end{split}
\end{equation*}
Let us now turn our attention to the estimate of $I_2$. Setting
\begin{equation*}
A_{k_0}=\bigcup_{k=k_0+1}^\infty\bigcup_i \widetilde{Q^k_i},
\end{equation*}
where $\widetilde{Q^k_i}=Q\big(x^k_i,\tau^{{p(k-k_0)}/n}(2\sqrt n)r^k_i\big)$ and $\tau$ is a fixed positive number such that $1<\tau<2$. Thus, we can further decompose $I_2$ as
\begin{equation*}
\begin{split}
I_2&\le\lambda^p\cdot \big|\big\{x\in A_{k_0}:|T_{\Omega}(F_2)(x)|>\lambda/2\big\}\big|+
\lambda^p\cdot \big|\big\{x\in (A_{k_0})^c:|T_{\Omega}(F_2)(x)|>\lambda/2\big\}\big|\\
&=I'_2+I''_2.
\end{split}
\end{equation*}
For the term $I'_2$, we can deduce that
\begin{equation*}
\begin{split}
I'_2&\le\lambda^p\sum_{k=k_0+1}^\infty\sum_i \big|\widetilde{Q^k_i}\big|\\
&\le C\cdot\lambda^p\sum_{k=k_0+1}^\infty\tau^{p(k-k_0)}\sum_i \big|Q^k_i\big|\\
&\le C\big\|f\big\|^p_{WH^p}\sum_{k=k_0+1}^\infty\Big(\frac{\tau}{\,2\,}\Big)^{p(k-k_0)}\\
&= C\big\|f\big\|^p_{WH^p}\sum_{k=1}^\infty\Big(\frac{\tau}{\,2\,}\Big)^{pk}\\
&\le C\big\|f\big\|^p_{WH^p}.
\end{split}
\end{equation*}
On the other hand, it follows directly from Chebyshev's inequality that
\begin{equation*}
\begin{split}
I''_2&\le 2^p\int_{(A_{k_0})^c}\big|T_{\Omega}(F_2)(x)\big|^p\,dx\\
&\le 2^p\sum_{k=k_0+1}^\infty\sum_i
\int_{\big(\widetilde{Q^k_i}\big)^c}\big|T_{\Omega}\big(b^k_i\big)(x)\big|^p\,dx.
\end{split}
\end{equation*}
Then, by the cancellation condition of $b^k_i\in L^\infty(\mathbb R^n)$, we get
\begin{equation*}
\begin{split}
\big|T_{\Omega}\big(b^k_i\big)(x)\big|&=
\left|\int_{Q^k_i}\bigg[\frac{\Omega(x,x-y)}{|x-y|^n}-\frac{\Omega(x,x-x^k_i)}{|x-x^k_i|^n}\bigg]
b^k_i(y)\,dy\right|\\
&\le C\int_{Q^k_i}\left|\frac{1}{|x-y|^n}-\frac{1}{|x-x^k_i|^n}\right|\big|b^k_i(y)\big|\,dy\\
&\ +\int_{Q^k_i}\frac{|\Omega(x,x-y)-\Omega(x,x-x^k_i)|}{|x-x^k_i|^n}\big|b^k_i(y)\big|\,dy\\
&= \mbox{\upshape I+II}.
\end{split}
\end{equation*}
Note that for any $y\in Q^k_i$ and $x\in\big(\widetilde{Q^k_i}\big)^c$, then $|x-y|\sim |x-x^k_i|$. This estimate together with the mean value theorem implies that
\begin{equation*}
\begin{split}
\mbox{\upshape I}&\leq C\int_{Q^k_i}\frac{|y-x^k_i|}{|x-x^k_i|^{n+1}}\big|b^k_i(y)\big|\,dy\\
&\leq C\cdot\big\|b^k_i\big\|_{L^\infty}\cdot\frac{(r^k_i)^{n+1}}{|x-x^k_i|^{n+1}}.
\end{split}
\end{equation*}
For the term II, we still have $|x-y|\sim |x-x^k_i|$, when $y\in Q^k_i$ and $x\in\big(\widetilde{Q^k_i}\big)^c$. Then we can readily see that
\begin{equation}\label{ineq-2}
\left|\frac{x-y}{|x-y|}-\frac{x-x^k_i}{|x-x^k_i|}\right|\leq C\cdot\frac{r^k_i}{|x-x^k_i|}.
\end{equation}
Hence, by the condition (\ref{L-logL}) and the inequality (\ref{ineq-2}), we deduce that for any $x\in\big(\widetilde{Q^k_i}\big)^c$,
\begin{align}\label{Omega-2}
\Big|\Omega(x,x-y)-\Omega(x,x-x^k_i)\Big|&=\left|\Omega\Big(x,\frac{x-y}{|x-y|}\Big)
-\Omega\Big(x,\frac{x-x^k_i}{|x-x^k_i|}\Big)\right|\notag\\
&\le\left(\frac{r^k_i}{|x-x^k_i|}\right)^{\sigma_1}\cdot\frac{C}{\big(\log\frac{|x-x^k_i|}{r^k_i}\big)^{\sigma_2}}.
\end{align}
So we have
\begin{equation*}
\mbox{\upshape II}\le C\cdot\big\|b^k_i\big\|_{L^\infty}
\cdot\frac{(r^k_i)^{n+\sigma_1}}{|x-x^k_i|^{n+\sigma_1}\big(\log\frac{|x-x^k_i|}{r^k_i}\big)^{\sigma_2}}.
\end{equation*}
Now denote $\tau^k_{i,\ell}=\big(\tau^{{p(k-k_0)}/n}\sqrt{n}\big)^{\ell}r^k_i$ and
$$E^k_{i,\ell}=\big\{x\in\mathbb R^n:\tau^k_{i,\ell}\le|x-x^k_i|<\tau^k_{i,\ell+1}\big\},\quad \ell=1,2,\ldots.$$
Then we split the term $I''_2$ by two parts,
\begin{equation*}
\begin{split}
I''_2&\le C\sum_{k=k_0+1}^\infty\sum_i\big\|b^k_i\big\|^p_{L^\infty}\big(r^k_i\big)^{p(n+1)}
\int_{\big(\widetilde{Q^k_i}\big)^c}\frac{dx}{|x-x^k_i|^{p(n+1)}}\\
&\ +C\sum_{k=k_0+1}^\infty\sum_i\big\|b^k_i\big\|^p_{L^\infty}\big(r^k_i\big)^{p(n+\sigma_1)}
\sum_{\ell=1}^\infty\int_{E^k_{i,\ell}}\frac{dx}{|x-x^k_i|^{p(n+\sigma_1)}\big(\log\frac{|x-x^k_i|}{r^k_i}\big)^{p\sigma_2}}\\
&=\mbox{\upshape III+IV}.
\end{split}
\end{equation*}
Let us consider the term III. Rewriting the above integral in polar coordinates and using the fact that $p(n+1)>n$, then we can get
\begin{equation*}
\begin{split}
\mbox{\upshape III}&\leq C\sum_{k=k_0+1}^\infty\sum_i\big\|b^k_i\big\|^p_{L^\infty}\big(r^k_i\big)^{p(n+1)}
\int_{|y|\geq \tau^{{p(k-k_0)}/n}\sqrt{n}r^k_i}\frac{dy}{|y|^{p(n+1)}}\\
&\leq C\sum_{k=k_0+1}^\infty\sum_i\big\|b^k_i\big\|^p_{L^\infty}\big(r^k_i\big)^{p(n+1)}
\int_{\tau^{{p(k-k_0)}/n}\sqrt{n}r^k_i}^{\infty}\frac{s^{n-1}}{s^{p(n+1)}}ds\\
&\leq C\sum_{k=k_0+1}^\infty\sum_i\big\|b^k_i\big\|^p_{L^\infty}\big(r^k_i\big)^{n}
\cdot\left[\frac{1}{\tau^{{p(k-k_0)}/n}}\right]^{p(n+1)-n}.
\end{split}
\end{equation*}
Recall that $\big\|b^k_i\big\|_{L^\infty}\le C2^k$ and $\sum_i\big|Q^k_i\big|\le C\cdot2^{-kp}\big\|f\big\|^p_{WH^p}$. Then
\begin{equation*}
\begin{split}
\mbox{\upshape III}&\leq C\sum_{k=k_0+1}^\infty 2^{kp}\cdot\left[\frac{1}{\tau^{{p(k-k_0)}/n}}\right]^{p(n+1)-n}\sum_i\big|Q^k_i\big|\\
&\leq C\big\|f\big\|^p_{WH^p}\sum_{k=k_0+1}^\infty\left[\frac{1}{\tau^{{p(k-k_0)}/n}}\right]^{p(n+1)-n}\\
&\leq C\big\|f\big\|^p_{WH^p}.
\end{split}
\end{equation*}
For the last term IV, we will also use the polar coordinates for integrals to obtain
\begin{equation*}
\begin{split}
\mbox{\upshape IV}&\leq C\sum_{k=k_0+1}^\infty\sum_i\big\|b^k_i\big\|^p_{L^\infty}\big(r^k_i\big)^{p(n+\sigma_1)}
\sum_{\ell=1}^\infty\int_{\tau^k_{i,\ell}\le|y|<\tau^k_{i,\ell+1}}
\frac{dy}{|y|^{p(n+\sigma_1)}\big(\log\frac{|y|}{r^k_i}\big)^{p\sigma_2}}\\
&\leq C\sum_{k=k_0+1}^\infty\sum_i\big\|b^k_i\big\|^p_{L^\infty}\big(r^k_i\big)^{p(n+\sigma_1)}
\sum_{\ell=1}^\infty\frac{1}{\big[(k-k_0)\log\tau\cdot\ell\big]^{p\sigma_2}}
\int_{\tau^k_{i,\ell}}^{\tau^k_{i,\ell+1}}\frac{s^{n-1}}{s^{p(n+\sigma_1)}}ds.
\end{split}
\end{equation*}

Let us now consider the following two cases:

$(i)$ $p(n+\sigma_1)=n$ and $p\sigma_2>2>1$.
\begin{equation*}
\begin{split}
\mbox{\upshape IV}&\leq C\sum_{k=k_0+1}^\infty\sum_i
\big\|b^k_i\big\|^p_{L^\infty}\big(r^k_i\big)^{n}\cdot\frac{(k-k_0)\log\tau}{[(k-k_0)\log\tau]^{p\sigma_2}}
\cdot\sum_{\ell=1}^\infty\frac{1}{\ell^{p\sigma_2}}\\
&\leq C\sum_{k=k_0+1}^\infty 2^{kp}\cdot\frac{(k-k_0)}{(k-k_0)^{p\sigma_2}}\sum_i\big|Q^k_i\big|\\
&\leq C\big\|f\big\|^p_{WH^p}\sum_{k=k_0+1}^\infty\frac{1}{(k-k_0)^{p\sigma_2-1}}\\
&\leq C\big\|f\big\|^p_{WH^p},
\end{split}
\end{equation*}
where in the last inequality we have used the fact that $p\sigma_2-1>1$.

$(ii)$ $p(n+\sigma_1)>n$ and $\sigma_2\geq0$. In this case, we have
\begin{equation*}
\begin{split}
\mbox{\upshape IV}&\leq C\sum_{k=k_0+1}^\infty\sum_i
\big\|b^k_i\big\|^p_{L^\infty}\big(r^k_i\big)^{n}\cdot\frac{(k-k_0)\log\tau}{[(k-k_0)\log\tau]^{p\sigma_2}}\\
&\times\sum_{\ell=1}^\infty\frac{1}{\ell^{p\sigma_2}}\left[\frac{1}{\tau^{{\ell p(k-k_0)}/n}}\right]^{p(n+\sigma_1)-n}\\
&\leq C\sum_{k=k_0+1}^\infty\sum_i
\big\|b^k_i\big\|^p_{L^\infty}\big(r^k_i\big)^{n}\cdot\frac{1}{(k-k_0)^{p\sigma_2-1}}\\
&\times\sum_{\ell=1}^\infty\left[\frac{1}{\tau^{{\ell p(k-k_0)}/n}}\right]^{p(n+\sigma_1)-n}.
\end{split}
\end{equation*}
Letting $\varepsilon=\frac{p[p(n+\sigma_1)-n]}{n}>0$. Since $\tau>1$, then we can easily see that
\begin{equation*}
\lim_{\ell\rightarrow\infty}\frac{\ell^{2+\varepsilon}}{\tau^{\ell\varepsilon}}=0.
\end{equation*}
Thus, for any $\ell\in\mathbb N_+$, there exists an absolute constant $C>0$ such that
\begin{equation*}
\frac{\ell^{2+\varepsilon}}{\tau^{\ell\varepsilon}}\leq C,\quad \ell=1,2,\ldots.
\end{equation*}
Therefore,
\begin{equation*}
\begin{split}
\mbox{\upshape IV}&\leq C\sum_{k=k_0+1}^\infty\sum_i
\big\|b^k_i\big\|^p_{L^\infty}\big(r^k_i\big)^{n}\cdot\frac{1}{(k-k_0)^{p\sigma_2-1}}\sum_{\ell=1}^\infty\frac{1}{[(k-k_0)\ell]^{2+\varepsilon}}\\
&\leq C\sum_{k=k_0+1}^\infty 2^{kp}\cdot\frac{1}{(k-k_0)^{p\sigma_2+1+\varepsilon}}\sum_i\big|Q^k_i\big|\\
&\leq C\big\|f\big\|^p_{WH^p}\sum_{k=k_0+1}^\infty\frac{1}{(k-k_0)^{p\sigma_2+1+\varepsilon}}\\
&\leq C\big\|f\big\|^p_{WH^p},
\end{split}
\end{equation*}
where the last inequality follows from the fact that $p\sigma_2+1+\varepsilon>1$.
Combining the above estimates for $I_1$ and $I_2$, and then taking the supremum over all $\lambda>0$, we complete the proof of Theorem \ref{mainthm:4}.
\end{proof}

\subsection{Proof of Theorem \ref{mainthm:5}}

\begin{proof}[Proof of Theorem $\ref{mainthm:5}$]
For any fixed $\lambda>0$, we may choose $k_0\in\mathbb Z$ satisfying $2^{k_0}\le\xi<2^{k_0+1}$, where we define $\xi=\lambda^{q/p}\big\|f\big\|^{1-q/p}_{WH^p}$. For every $f\in WH^p(\mathbb R^n)$, then in view of Theorem \ref{thm:weak Hardy}, we can write
\begin{equation*}
f=\sum_{k=-\infty}^\infty f_k=\sum_{k=-\infty}^{k_0} f_k+\sum_{k=k_0+1}^\infty f_k:=F_1+F_2,
\end{equation*}
where $F_1=\sum_{k=-\infty}^{k_0} f_k=\sum_{k=-\infty}^{k_0}\sum_i b^k_i$, $F_2=\sum_{k=k_0+1}^\infty f_k=\sum_{k=k_0+1}^\infty\sum_i b^k_i$ and $\{b^k_i\}$ satisfies $(a)$--$(c)$ in Theorem \ref{thm:weak Hardy}. Then we have
\begin{equation*}
\begin{split}
&\lambda\cdot \big|\big\{x\in\mathbb R^n:\big|T_{\Omega,\alpha}(f)(x)\big|>\lambda\big\}\big|^{1/q}\\
\leq\ &\lambda\cdot \big|\big\{x\in\mathbb R^n:\big|T_{\Omega,\alpha}(F_1)(x)\big|>\lambda/2\big\}\big|^{1/q}
+\lambda\cdot \big|\big\{x\in\mathbb R^n:\big|T_{\Omega,\alpha}(F_2)(x)\big|>\lambda/2\big\}\big|^{1/q}\\
:=\ &J_1+J_2.
\end{split}
\end{equation*}
For $0<\alpha<n$, we are able to choose $p_1$ and $q_1>p_1$ such that $1<p_1<n/{\alpha}$ and $1/{q_1}=1/{p_1}-\alpha/n$. Similar to the proof of Theorem \ref{mainthm:4}, we first claim that the following inequality holds:
\begin{equation}\label{ineq-3}
\big\|F_1\big\|_{L^{p_1}}\le C\cdot\xi^{1-p/{p_1}}\big\|f\big\|^{p/{p_1}}_{WH^p}.
\end{equation}
Indeed, since supp\,$b^k_i\subseteq Q^k_i=Q\big(x^k_i,r^k_i\big)$ and $\big\|b^k_i\big\|_{L^\infty}\le C 2^k$ by Theorem \ref{thm:weak Hardy}, then by using Minkowski's inequality for integrals, we get
\begin{equation*}
\big\|F_1\big\|_{L^{p_1}}\leq\sum_{k=-\infty}^{k_0}\Big\|\sum_i b^k_i\Big\|_{L^{p_1}}.
\end{equation*}
For each $k\in\mathbb Z$, by using the finitely overlapping property of the cubes $\{Q^k_i\}$ and the fact that $1-p/{p_1}>0$, we thus obtain
\begin{equation*}
\begin{split}
\Big\|\sum_i b^k_i\Big\|_{L^{p_1}}&\leq\sup_i\big\|b^k_i\big\|_{L^\infty}\left(\int_{x\in\cup_i Q^k_i}1\,dx\right)^{1/{p_1}}\\
&\leq C\cdot2^k\bigg(\sum_i \big|Q^k_i\big|\bigg)^{1/{p_1}}\\
&\leq C\cdot2^{k(1-p/{p_1})}\big\|f\big\|^{p/{p_1}}_{WH^p},
\end{split}
\end{equation*}
and
\begin{equation*}
\begin{split}
\big\|F_1\big\|_{L^{p_1}}&\leq C\sum_{k=-\infty}^{k_0}2^{(k-k_0)(1-p/{p_1})}\cdot\xi^{1-p/{p_1}}\big\|f\big\|^{p/{p_1}}_{WH^p}\\
&=C\cdot\xi^{1-p/{p_1}}\big\|f\big\|^{p/{p_1}}_{WH^p}\sum_{k=0}^\infty\left(\frac{\,1\,}{2}\right)^{k(1-p/{p_1})}\\
&\leq C\cdot\xi^{1-p/{p_1}}\big\|f\big\|^{p/{p_1}}_{WH^p}.
\end{split}
\end{equation*}
By our assumption, we know that $T_{\Omega,\alpha}$ is bounded from $L^{p_1}(\mathbb R^n)$ to $L^{q_1}(\mathbb R^n)$ according to Theorem B. This fact along with Chebyshev's inequality and the inequality (\ref{ineq-3}) implies that
\begin{equation*}
\begin{split}
J_1&\le\lambda\cdot\Big(\frac{2}{\lambda}\Big)^{{q_1}/q}
\Big(\big\|T_{\Omega,\alpha}(F_1)\big\|_{L^{q_1}}\Big)^{{q_1}/q}\\
&\le C\cdot\lambda^{1-{q_1}/q}\Big(\big\|F_1\big\|_{L^{p_1}}\Big)^{{q_1}/q}\\
&\le C\cdot\lambda^{1-{q_1}/q}\xi^{(1-p/{p_1})\cdot{q_1}/q}\big\|f\big\|^{{(pq_1)}/{(p_1 q)}}_{WH^p}\\
&\le C\cdot\lambda^{1-{q_1}/q}\Big(\lambda^{q/p}\big\|f\big\|^{1-q/p}_{WH^p}\Big)^{(1-p/{p_1})\cdot{q_1}/q}
\big\|f\big\|^{{(pq_1)}/{(p_1 q)}}_{WH^p}.
\end{split}
\end{equation*}
Notice that $1/p-1/q=1/{p_1}-1/{q_1}=\alpha/n$, then it is easy to check that
\begin{equation*}
\begin{split}
&1-{q_1}/q+q/p\cdot(1-p/{p_1})\cdot{q_1}/q\\
=&1-{q_1}/q+q_1\cdot(1/p-1/{p_1})\\
=&1-{q_1}/q+q_1\cdot(1/q-1/{q_1})\\
=&0
\end{split}
\end{equation*}
and
\begin{equation*}
\begin{split}
&(1-q/p)(1-p/{p_1})\cdot{q_1}/q+{(pq_1)}/{(p_1 q)}\\
=&(1-q/p)(1-p/{p_1})\cdot{q_1}/q-(1-p/{p_1})\cdot{q_1}/q+{q_1}/q\\
=&q/p\cdot(p/{p_1}-1)\cdot{q_1}/q+{q_1}/q\\
=&q_1(1/{q_1}-1/q)+{q_1}/q\\
=&1.
\end{split}
\end{equation*}
Hence
\begin{equation*}
J_1\le C\big\|f\big\|_{WH^p}.
\end{equation*}
We now turn our attention to the estimate of $J_2$. Setting
\begin{equation*}
A_{k_0}=\bigcup_{k=k_0+1}^\infty\bigcup_i \widetilde{Q^k_i},
\end{equation*}
where $\widetilde{Q^k_i}=Q\big(x^k_i,\tau^{{p(k-k_0)}/n}(2\sqrt n)r^k_i\big)$ and $\tau$ is also a fixed positive number such that $1<\tau<2$. Thus, we can further split $J_2$ into two parts,
\begin{equation*}
\begin{split}
J_2&\le\lambda\cdot \big|\big\{x\in A_{k_0}:|T_{\Omega,\alpha}(F_2)(x)|>\lambda/2\big\}\big|^{1/q}\\
&+
\lambda\cdot \big|\big\{x\in (A_{k_0})^c:|T_{\Omega,\alpha}(F_2)(x)|>\lambda/2\big\}\big|^{1/q}\\
&=J'_2+J''_2.
\end{split}
\end{equation*}
For the term $J'_2$, we can see that
\begin{equation*}
\begin{split}
J'_2&\leq\lambda\bigg(\sum_{k=k_0+1}^\infty\sum_i \big|\widetilde{Q^k_i}\big|\bigg)^{1/q}\\
&\leq C\cdot\lambda\bigg(\sum_{k=k_0+1}^\infty\tau^{p(k-k_0)}\sum_i \big|Q^k_i\big|\bigg)^{1/q}\\
&\leq C\cdot\lambda\cdot\xi^{-p/q}
\big\|f\big\|_{WH^p}^{p/q}\bigg(\sum_{k=k_0+1}^\infty\Big[\frac{\tau}{\,2\,}\Big]^{{p(k-k_0)}}\bigg)^{1/q}\\
&\leq C\cdot\lambda\cdot\Big(\lambda^{q/p}\big\|f\big\|^{1-q/p}_{WH^p}\Big)^{-p/q}\big\|f\big\|_{WH^p}^{p/q}\\
&=  C\big\|f\big\|_{WH^p}.
\end{split}
\end{equation*}
For the term $J''_2$, note that $n/{(n-\alpha+1)}<q\leq n/{(n-\alpha)}$ when $n/{(n+1)}<p\leq1$ and $1/q=1/p-\alpha/n$. For the case of $q>1$, by using Chebyshev's inequality and Minkowski's inequality for integrals, we have
\begin{equation*}
\begin{split}
J''_2&\le 2\left(\int_{(A_{k_0})^c}\big|T_{\Omega,\alpha}(F_2)(x)\big|^q\,dx\right)^{1/q}\\
&\le 2\sum_{k=k_0+1}^\infty\sum_i
\left(\int_{\big(\widetilde{Q^k_i}\big)^c}\big|T_{\Omega,\alpha}\big(b^k_i\big)(x)\big|^q\,dx\right)^{1/q}.
\end{split}
\end{equation*}
On the other hand, for the case of $q\leq1$, then by using Chebyshev's inequality and the well-known inequality $(\sum_i\mu_i)^{q}\leq\sum_i(\mu_i)^{q}$, we conclude that
\begin{equation*}
\begin{split}
J''_2&\le 2\left(\int_{(A_{k_0})^c}\big|T_{\Omega,\alpha}(F_2)(x)\big|^q\,dx\right)^{1/q}\\
&\le 2\left(\sum_{k=k_0+1}^\infty\sum_i\int_{\big(\widetilde{Q^k_i}\big)^c}\big|T_{\Omega,\alpha}\big(b^k_i\big)(x)\big|^q\,dx\right)^{1/q}.
\end{split}
\end{equation*}
Furthermore, by the cancellation condition of $b^k_i\in L^\infty(\mathbb R^n)$, we get
\begin{equation*}
\begin{split}
\big|T_{\Omega,\alpha}\big(b^k_i\big)(x)\big|&=
\left|\int_{Q^k_i}\bigg[\frac{\Omega(x,x-y)}{|x-y|^{n-\alpha}}-\frac{\Omega(x,x-x^k_i)}{|x-x^k_i|^{n-\alpha}}\bigg]
b^k_i(y)\,dy\right|\\
&\le C\int_{Q^k_i}\left|\frac{1}{|x-y|^{n-\alpha}}-\frac{1}{|x-x^k_i|^{n-\alpha}}\right|
\big|b^k_i(y)\big|\,dy\\
&\ +\int_{Q^k_i}\frac{|\Omega(x,x-y)-\Omega(x,x-x^k_i)|}{|x-x^k_i|^{n-\alpha}}\big|b^k_i(y)\big|\,dy\\
&= \mbox{\upshape I+II}.
\end{split}
\end{equation*}
Let us consider the term I. Noting that if $y\in Q^k_i$ and $x\in\big(\widetilde{Q^k_i}\big)^c$, then we have $|x-y|\sim |x-x^k_i|$. Using the mean value theorem again, we obtain
\begin{equation*}
\begin{split}
\mbox{\upshape I}&\leq C\int_{Q^k_i}\frac{|y-x^k_i|}{|x-x^k_i|^{n-\alpha+1}}\big|b^k_i(y)\big|\,dy\\
&\leq C\cdot\big\|b^k_i\big\|_{L^\infty}\cdot\frac{(r^k_i)^{n+1}}{|x-x^k_i|^{n-\alpha+1}}.
\end{split}
\end{equation*}
For the other term II, we still have $|x-y|\sim |x-x^k_i|$ for all $i$ and $k$, when $y\in Q^k_i$ and $x\in\big(\widetilde{Q^k_i}\big)^c$. Hence, it follows from the previous inequalities (\ref{ineq-2}) and (\ref{Omega-2}) that
\begin{equation*}
\mbox{\upshape II}\le C\cdot\big\|b^k_i\big\|_{L^\infty}\cdot\frac{(r^k_i)^{n+\sigma_1}}{|x-x^k_i|^{n-\alpha+\sigma_1}
\big(\log\frac{|x-x^k_i|}{r^k_i}\big)^{\sigma_2}}.
\end{equation*}
We denote $\tau^k_{i,\ell}=\big(\tau^{{p(k-k_0)}/n}\sqrt{n}\big)^{\ell}r^k_i$ and
$$E^k_{i,\ell}=\big\{x\in\mathbb R^n:\tau^k_{i,\ell}\le|x-x^k_i|<\tau^k_{i,\ell+1}\big\},\quad \ell=1,2,\ldots.$$
Consequently, for the case of $q>1$, we have
\begin{equation*}
\begin{split}
J''_2&\le C\sum_{k=k_0+1}^\infty\sum_i\big\|b^k_i\big\|_{L^\infty}\big(r^k_i\big)^{n+1}
\left(\int_{\big(\widetilde{Q^k_i}\big)^c}\frac{dx}{|x-x^k_i|^{q(n-\alpha+1)}}\right)^{1/q}\\
&\ +C\sum_{k=k_0+1}^\infty\sum_i\big\|b^k_i\big\|_{L^\infty}\big(r^k_i\big)^{n+\sigma_1}
\left(\sum_{\ell=1}^\infty\int_{E^k_{i,\ell}}\frac{dx}{|x-x^k_i|^{q(n-\alpha+\sigma_1)}
\big(\log\frac{|x-x^k_i|}{r^k_i}\big)^{q\sigma_2}}\right)^{1/q}\\
&=\mbox{\upshape III}^{(1)}+\mbox{\upshape IV}^{(1)}.
\end{split}
\end{equation*}
Notice that $q(n-\alpha+1)>n$ and $1/q=1/p-\alpha/n$. We then use the polar coordinates for integrals to obtain
\begin{equation*}
\begin{split}
\mbox{\upshape III}^{(1)}&\leq C\sum_{k=k_0+1}^\infty\sum_i\big\|b^k_i\big\|_{L^\infty}\big(r^k_i\big)^{n+1}
\left(\int_{|y|\geq \tau^{{p(k-k_0)}/n}\sqrt{n}r^k_i}\frac{dy}{|y|^{q(n-\alpha+1)}}\right)^{1/q}\\
&\leq C\sum_{k=k_0+1}^\infty\sum_i\big\|b^k_i\big\|_{L^\infty}\big(r^k_i\big)^{n+1}
\left(\int_{\tau^{{p(k-k_0)}/n}\sqrt{n}r^k_i}^{\infty}\frac{s^{n-1}}{s^{q(n-\alpha+1)}}ds\right)^{1/q}\\
&\leq C\sum_{k=k_0+1}^\infty\sum_i\big\|b^k_i\big\|_{L^\infty}\big(r^k_i\big)^{n+1}
\cdot\left[\frac{1}{\tau^{{p(k-k_0)}/n}\cdot r^k_i}\right]^{n-\alpha+1-n/q}\\
&\leq C\sum_{k=k_0+1}^\infty 2^k\cdot\left[\frac{1}{\tau^{{p(k-k_0)}/n}}\right]^{n-\alpha+1-n/q}\sum_i\big|Q^k_i\big|^{1/p}.
\end{split}
\end{equation*}
Since $1/p\geq1$, by using the well-known inequality $\sum_i(\mu_i)^{1/p}\le(\sum_i\mu_i)^{1/p}$ and $\sum_i\big|Q^k_i\big|\le C\cdot2^{-kp}\big\|f\big\|^p_{WH^p}$, we have
\begin{equation*}
\begin{split}
\mbox{\upshape III}^{(1)}&\leq C\sum_{k=k_0+1}^\infty 2^k\cdot\left[\frac{1}{\tau^{{p(k-k_0)}/n}}\right]^{n-\alpha+1-n/q}
\left(\sum_i\big|Q^k_i\big|\right)^{1/p}\\
&\leq C\big\|f\big\|_{WH^p}\sum_{k=k_0+1}^\infty\left[\frac{1}{\tau^{{p(k-k_0)}/n}}\right]^{n-\alpha+1-n/q}\\
&\leq C\big\|f\big\|_{WH^p}.
\end{split}
\end{equation*}
In order to estimate the last term $\mbox{\upshape IV}^{(1)}$, we will use the polar coordinates for integrals again to obtain
\begin{equation*}
\begin{split}
\mbox{\upshape IV}^{(1)}&\leq C\sum_{k=k_0+1}^\infty\sum_i\big\|b^k_i\big\|_{L^\infty}\big(r^k_i\big)^{n+\sigma_1}
\left(\sum_{\ell=1}^\infty\int_{\tau^k_{i,\ell}\le|y|<\tau^k_{i,\ell+1}}
\frac{dy}{|y|^{q(n-\alpha+\sigma_1)}\big(\log\frac{|y|}{r^k_i}\big)^{q\sigma_2}}\right)^{1/q}\\
&\leq C\sum_{k=k_0+1}^\infty\sum_i\big\|b^k_i\big\|_{L^\infty}\big(r^k_i\big)^{n+\sigma_1}
\left(\sum_{\ell=1}^\infty\frac{1}{\big[(k-k_0)\log\tau\cdot\ell\big]^{q\sigma_2}}
\int_{\tau^k_{i,\ell}}^{\tau^k_{i,\ell+1}}\frac{s^{n-1}}{s^{q(n-\alpha+\sigma_1)}}ds\right)^{1/q}.
\end{split}
\end{equation*}

We are going to consider the following two cases:

$(i)$ $q(n-\alpha+\sigma_1)=n$ and $\sigma_2>1/q+1$. Then we know that $n+\sigma_1=n/p$. Hence, by using the well-known inequality $\sum_i(\mu_i)^{1/p}\le(\sum_i\mu_i)^{1/p}$ and $\sum_i\big|Q^k_i\big|\le C\cdot2^{-kp}\big\|f\big\|^p_{WH^p}$ again, we have
\begin{equation*}
\begin{split}
\mbox{\upshape IV}^{(1)}&\leq C\sum_{k=k_0+1}^\infty\sum_i
\big\|b^k_i\big\|_{L^\infty}\big(r^k_i\big)^{n/p}\cdot\frac{[(k-k_0)\log\tau]^{1/q}}{[(k-k_0)\log\tau]^{\sigma_2}}
\left(\sum_{\ell=1}^\infty\frac{1}{\ell^{q\sigma_2}}\right)^{1/q}\\
&\leq C\sum_{k=k_0+1}^\infty 2^k\cdot\frac{(k-k_0)^{1/q}}{(k-k_0)^{\sigma_2}}
\sum_i\big|Q^k_i\big|^{1/p}\\
\end{split}
\end{equation*}
\begin{equation*}
\begin{split}
&\leq C\sum_{k=k_0+1}^\infty 2^k\cdot\frac{(k-k_0)^{1/q}}{(k-k_0)^{\sigma_2}}
\left(\sum_i\big|Q^k_i\big|\right)^{1/p}\\
&\leq C\big\|f\big\|_{WH^p}\sum_{k=k_0+1}^\infty\frac{1}{(k-k_0)^{\sigma_2-1/q}}\\
&\leq C\big\|f\big\|_{WH^p}.
\end{split}
\end{equation*}

$(ii)$ $q(n-\alpha+\sigma_1)>n$ and $\sigma_2\geq 0$. In this case, we have
\begin{equation*}
\begin{split}
\mbox{\upshape IV}^{(1)}&\leq C\sum_{k=k_0+1}^\infty\sum_i
\big\|b^k_i\big\|_{L^\infty}\big(r^k_i\big)^{n+\sigma_1}\cdot\frac{[(k-k_0)\log\tau]^{1/q}}{[(k-k_0)\log\tau]^{\sigma_2}}\\
&\times\left(\sum_{\ell=1}^\infty\frac{1}{\ell^{q\sigma_2}}\left[\frac{1}{\tau^{{\ell p(k-k_0)}/n}\cdot r^k_i}\right]^{q(n-\alpha+\sigma_1)-n}\right)^{1/q}\\
&\leq C\sum_{k=k_0+1}^\infty\sum_i
\big\|b^k_i\big\|_{L^\infty}\big(r^k_i\big)^{\alpha+n/q}\cdot\frac{1}{(k-k_0)^{\sigma_2-1/q}}\\
&\times\left(\sum_{\ell=1}^\infty\left[\frac{1}{\tau^{{\ell p(k-k_0)}/n}}\right]^{q(n-\alpha+\sigma_1)-n}\right)^{1/q}.
\end{split}
\end{equation*}
Letting $\varepsilon'=\frac{p[q(n-\alpha+\sigma_1)-n]}{n}>0$. Since $\tau>1$, it is easy to verify that
\begin{equation*}
\lim_{\ell\rightarrow\infty}\frac{\ell^{q+1+\varepsilon'}}{\tau^{\ell\varepsilon'}}=0.
\end{equation*}
Thus, for any $\ell\in\mathbb N_+$, there exists an absolute constant $C>0$ such that
\begin{equation*}
\frac{\ell^{q+1+\varepsilon'}}{\tau^{\ell\varepsilon'}}\leq C,\quad \ell=1,2,\ldots.
\end{equation*}
Therefore,
\begin{equation*}
\begin{split}
\mbox{\upshape IV}^{(1)}&\leq C\sum_{k=k_0+1}^\infty\sum_i
\big\|b^k_i\big\|_{L^\infty}\big(r^k_i\big)^{n/p}\cdot\frac{1}{(k-k_0)^{\sigma_2-1/q}}\\
&\times\left(\sum_{\ell=1}^\infty\frac{1}{[(k-k_0)\ell]^{q+1+\varepsilon'}}\right)^{1/q}\\
&\leq C\sum_{k=k_0+1}^\infty 2^{k}\cdot\frac{1}{(k-k_0)^{\sigma_2+1+{\varepsilon'}/q}}\sum_i\big|Q^k_i\big|^{1/p}\\
&\leq C\big\|f\big\|_{WH^p}\sum_{k=k_0+1}^\infty\frac{1}{(k-k_0)^{\sigma_2+1+{\varepsilon'}/q}}\\
&\leq C\big\|f\big\|_{WH^p},
\end{split}
\end{equation*}
where the last inequality follows from our assumption that $\sigma_2+1+{\varepsilon'}/q>1$. On the other hand, for the case of $q\leq1$, we have
\begin{equation*}
\begin{split}
J''_2&\le C\left(\sum_{k=k_0+1}^\infty\sum_i\big\|b^k_i\big\|^q_{L^\infty}\big(r^k_i\big)^{q(n+1)}
\int_{\big(\widetilde{Q^k_i}\big)^c}\frac{dx}{|x-x^k_i|^{q(n-\alpha+1)}}\right)^{1/q}\\
&\ +C\left(\sum_{k=k_0+1}^\infty\sum_i\big\|b^k_i\big\|^q_{L^\infty}\big(r^k_i\big)^{q(n+\sigma_1)}
\sum_{\ell=1}^\infty\int_{E^k_{i,\ell}}\frac{dx}{|x-x^k_i|^{q(n-\alpha+\sigma_1)}
\big(\log\frac{|x-x^k_i|}{r^k_i}\big)^{q\sigma_2}}\right)^{1/q}\\
&=\mbox{\upshape III}^{(2)}+\mbox{\upshape IV}^{(2)}.
\end{split}
\end{equation*}
For the term $\mbox{\upshape III}^{(2)}$, note that $q(n-\alpha+1)>n$ and $n/q=n/p-\alpha$. Making use of the polar coordinates for integrals, we find that
\begin{equation*}
\begin{split}
\mbox{\upshape III}^{(2)}&\leq C\left(\sum_{k=k_0+1}^\infty\sum_i\big\|b^k_i\big\|^q_{L^\infty}\big(r^k_i\big)^{q(n+1)}
\int_{|y|\geq \tau^{{p(k-k_0)}/n}\sqrt{n}r^k_i}\frac{dy}{|y|^{q(n-\alpha+1)}}\right)^{1/q}\\
&\leq C\left(\sum_{k=k_0+1}^\infty\sum_i\big\|b^k_i\big\|^q_{L^\infty}\big(r^k_i\big)^{q(n+1)}
\int_{\tau^{{p(k-k_0)}/n}\sqrt{n}r^k_i}^{\infty}\frac{s^{n-1}}{s^{q(n-\alpha+1)}}ds\right)^{1/q}\\
&\leq C\left(\sum_{k=k_0+1}^\infty\sum_i\big\|b^k_i\big\|^q_{L^\infty}\big(r^k_i\big)^{q(n+1)}
\cdot\left[\frac{1}{\tau^{{p(k-k_0)}/n}\cdot r^k_i}\right]^{q(n-\alpha+1)-n}\right)^{1/q}\\
&\leq C\left(\sum_{k=k_0+1}^\infty 2^{kq}\cdot\left[\frac{1}{\tau^{{p(k-k_0)}/n}}\right]^{q(n-\alpha+1)-n}
\sum_i\big|Q^k_i\big|^{q/p}\right)^{1/q}.
\end{split}
\end{equation*}
Since $q/p>1$, by using the well-known inequality $\sum_i(\mu_i)^{q/p}\le(\sum_i\mu_i)^{q/p}$ and $\sum_i\big|Q^k_i\big|\le C\cdot2^{-kp}\big\|f\big\|^p_{WH^p}$, we have
\begin{equation*}
\begin{split}
\mbox{\upshape III}^{(2)}&\leq C\left(\sum_{k=k_0+1}^\infty 2^{kq}\cdot\left[\frac{1}{\tau^{{p(k-k_0)}/n}}\right]^{q(n-\alpha+1)-n}
\left[\sum_i\big|Q^k_i\big|\right]^{q/p}\right)^{1/q}\\
&\leq C\left(\sum_{k=k_0+1}^\infty\big\|f\big\|^q_{WH^p}\cdot\left[\frac{1}{\tau^{{p(k-k_0)}/n}}\right]^{q(n-\alpha+1)-n}\right)^{1/q}\\
&\leq C\big\|f\big\|_{WH^p}.
\end{split}
\end{equation*}
For the last term $\mbox{\upshape IV}^{(2)}$, we will use the polar coordinates for integrals again to obtain
\begin{equation*}
\begin{split}
\mbox{\upshape IV}^{(2)}&\leq C\left(\sum_{k=k_0+1}^\infty\sum_i\big\|b^k_i\big\|^q_{L^\infty}\big(r^k_i\big)^{q(n+\sigma_1)}
\sum_{\ell=1}^\infty\int_{\tau^k_{i,\ell}\le|y|<\tau^k_{i,\ell+1}}
\frac{dy}{|y|^{q(n-\alpha+\sigma_1)}\big(\log\frac{|y|}{r^k_i}\big)^{q\sigma_2}}\right)^{1/q}\\
&\leq C\left(\sum_{k=k_0+1}^\infty\sum_i\big\|b^k_i\big\|^q_{L^\infty}\big(r^k_i\big)^{q(n+\sigma_1)}
\sum_{\ell=1}^\infty\frac{1}{\big[(k-k_0)\log\tau\cdot\ell\big]^{q\sigma_2}}
\int_{\tau^k_{i,\ell}}^{\tau^k_{i,\ell+1}}\frac{s^{n-1}}{s^{q(n-\alpha+\sigma_1)}}ds\right)^{1/q}.
\end{split}
\end{equation*}

We are going to discuss the following two cases:

$(i)$ $q(n-\alpha+\sigma_1)=n$ and $q\sigma_2>2>1$. Then we know that $n+\sigma_1=n/p$. Thus, by using the well-known inequality $\sum_i(\mu_i)^{q/p}\le(\sum_i\mu_i)^{q/p}$ and $\sum_i\big|Q^k_i\big|\le C\cdot2^{-kp}\big\|f\big\|^p_{WH^p}$ again, we get
\begin{equation*}
\begin{split}
\mbox{\upshape IV}^{(2)}&\leq C\left(\sum_{k=k_0+1}^\infty\sum_i
\big\|b^k_i\big\|^q_{L^\infty}\big(r^k_i\big)^{{nq}/p}\cdot\frac{[(k-k_0)\log\tau]}{[(k-k_0)\log\tau]^{q\sigma_2}}
\sum_{\ell=1}^\infty\frac{1}{\ell^{q\sigma_2}}\right)^{1/q}\\
&\leq C\left(\sum_{k=k_0+1}^\infty 2^{kq}\cdot\frac{(k-k_0)}{(k-k_0)^{q\sigma_2}}
\sum_i\big|Q^k_i\big|^{q/p}\right)^{1/q}\\
&\leq C\left(\sum_{k=k_0+1}^\infty 2^{kq}\cdot\frac{(k-k_0)}{(k-k_0)^{q\sigma_2}}
\left[\sum_i\big|Q^k_i\big|\right]^{q/p}\right)^{1/q}\\
&\leq C\big\|f\big\|_{WH^p}\left(\sum_{k=k_0+1}^\infty\frac{1}{(k-k_0)^{q\sigma_2-1}}\right)^{1/q}\\
&\leq C\big\|f\big\|_{WH^p}.
\end{split}
\end{equation*}

$(ii)$ $q(n-\alpha+\sigma_1)>n$ and $\sigma_2\geq 0$. In the present situation, we have
\begin{equation*}
\begin{split}
\mbox{\upshape IV}^{(2)}&\leq C\Bigg(\sum_{k=k_0+1}^\infty\sum_i
\big\|b^k_i\big\|^q_{L^\infty}\big(r^k_i\big)^{q(n+\sigma_1)}\cdot\frac{[(k-k_0)\log\tau]}{[(k-k_0)\log\tau]^{q\sigma_2}}\\
&\times\sum_{\ell=1}^\infty\frac{1}{\ell^{q\sigma_2}}\left[\frac{1}{\tau^{{\ell p(k-k_0)}/n}\cdot r^k_i}\right]^{q(n-\alpha+\sigma_1)-n}\Bigg)^{1/q}\\
&\leq C\Bigg(\sum_{k=k_0+1}^\infty\sum_i
\big\|b^k_i\big\|^q_{L^\infty}\big(r^k_i\big)^{q\alpha+n}\cdot\frac{1}{(k-k_0)^{q\sigma_2-1}}\\
&\times\sum_{\ell=1}^\infty\left[\frac{1}{\tau^{{\ell p(k-k_0)}/n}}\right]^{q(n-\alpha+\sigma_1)-n}\Bigg)^{1/q}.
\end{split}
\end{equation*}
Letting $\varepsilon''=\frac{p[q(n-\alpha+\sigma_1)-n]}{n}>0$. Since $\tau>1$, it is easy to see that
\begin{equation*}
\lim_{\ell\rightarrow\infty}\frac{\ell^{2+\varepsilon''}}{\tau^{\ell\varepsilon''}}=0.
\end{equation*}
Thus, for any $\ell\in\mathbb N_+$, there exists an absolute constant $C>0$ such that
\begin{equation*}
\frac{\ell^{2+\varepsilon''}}{\tau^{\ell\varepsilon''}}\leq C,\quad \ell=1,2,\ldots.
\end{equation*}
Taking into account the facts that $q\alpha+n={nq}/p$ and $q/p>1$, we have eventually obtain
\begin{equation*}
\begin{split}
\mbox{\upshape IV}^{(2)}&\leq C\Bigg(\sum_{k=k_0+1}^\infty\sum_i
\big\|b^k_i\big\|^q_{L^\infty}\big(r^k_i\big)^{{nq}/p}\cdot\frac{1}{(k-k_0)^{q\sigma_2-1}}\\
&\times\sum_{\ell=1}^\infty\frac{1}{[(k-k_0)\ell]^{2+\varepsilon''}}\Bigg)^{1/q}\\
&\leq C\left(\sum_{k=k_0+1}^\infty 2^{kq}\cdot\frac{1}{(k-k_0)^{q\sigma_2+1+\varepsilon''}}\sum_i\big|Q^k_i\big|^{q/p}\right)^{1/q}\\
&\leq C\left(\sum_{k=k_0+1}^\infty 2^{kq}\cdot\frac{1}{(k-k_0)^{q\sigma_2+1+\varepsilon''}}
\left[\sum_i\big|Q^k_i\big|\right]^{q/p}\right)^{1/q}\\
&\leq C\big\|f\big\|_{WH^p}\left(\sum_{k=k_0+1}^\infty\frac{1}{(k-k_0)^{q\sigma_2+1+\varepsilon''}}\right)^{1/q}\\
&\leq C\big\|f\big\|_{WH^p},
\end{split}
\end{equation*}
where the last inequality is due to $q\sigma_2+1+\varepsilon''>1$. Collecting all the above estimates and then taking the supremum over all $\lambda>0$, we finish the proof of Theorem \ref{mainthm:5}.
\end{proof}

\subsection{Proof of Theorem \ref{mainthm:6}}

\begin{proof}[Proof of Theorem $\ref{mainthm:6}$]
Arguing as in the proof of Theorem \ref{mainthm:4}, for any fixed $\lambda>0$, we can choose $k_0\in\mathbb Z$ satisfying $2^{k_0}\leq\lambda<2^{k_0+1}$. For every $f\in WH^p(\mathbb R^n)$, then in view of Theorem \ref{thm:weak Hardy}, we may write
\begin{equation*}
f=\sum_{k=-\infty}^\infty f_k=\sum_{k=-\infty}^{k_0} f_k+\sum_{k=k_0+1}^\infty f_k:=F_1+F_2,
\end{equation*}
where $F_1=\sum_{k=-\infty}^{k_0} f_k=\sum_{k=-\infty}^{k_0}\sum_i b^k_i$, $F_2=\sum_{k=k_0+1}^\infty f_k=\sum_{k=k_0+1}^\infty\sum_i b^k_i$ and $\{b^k_i\}$ satisfies $(a)$--$(c)$ in Theorem \ref{thm:weak Hardy}. Then we have
\begin{equation*}
\begin{split}
&\lambda^p\cdot \big|\big\{x\in\mathbb R^n:\big|\mu^{\rho}_{\Omega}(f)(x)\big|>\lambda\big\}\big|\\
\leq\ &\lambda^p\cdot \big|\big\{x\in\mathbb R^n:\big|\mu^{\rho}_{\Omega}(F_1)(x)\big|>\lambda/2\big\}\big|
+\lambda^p\cdot \big|\big\{x\in\mathbb R^n:\big|\mu^{\rho}_{\Omega}(F_2)(x)\big|>\lambda/2\big\}\big|\\
:=\ &K_1+K_2.
\end{split}
\end{equation*}
Applying Chebyshev's inequality, Theorem \ref{Mar:L2} and the inequality (\ref{F1:L2}), we get
\begin{equation*}
\begin{split}
K_1&\leq \lambda^p\cdot\frac{4}{\lambda^2}\big\|\mu^{\rho}_{\Omega}(F_1)\big\|^2_{L^2}\notag\\
&\leq C\cdot\lambda^{p-2}\big\|F_1\big\|^2_{L^2}\notag\\
&\leq C\big\|f\big\|^p_{WH^p}.
\end{split}
\end{equation*}
Let us now consider the other term $K_2$. As before, we set
\begin{equation*}
A_{k_0}=\bigcup_{k=k_0+1}^\infty\bigcup_i \widetilde{Q^k_i},
\end{equation*}
where $\widetilde{Q^k_i}=Q\big(x^k_i,\tau^{{p(k-k_0)}/n}(2\sqrt n)r^k_i\big)$ and $\tau$ is an appropriately chosen number such that $1<\tau<2$. Thus, we can further decompose $K_2$ as
\begin{equation*}
\begin{split}
K_2&\leq\lambda^p\cdot \big|\big\{x\in A_{k_0}:\big|\mu^{\rho}_{\Omega}(F_2)(x)\big|>\lambda/2\big\}\big|
+\lambda^p\cdot \big|\big\{x\in (A_{k_0})^c:\big|\mu^{\rho}_{\Omega}(F_2)(x)\big|>\lambda/2\big\}\big|\\
&=K'_2+K''_2.
\end{split}
\end{equation*}
By using the same procedure as in Theorem \ref{mainthm:4}, we can also obtain
\begin{equation*}
K'_2\le C\big\|f\big\|^p_{WH^p}.
\end{equation*}
It remains to estimate the last term $K''_2$. We first apply Chebyshev's inequality to obtain
\begin{equation*}
\begin{split}
K''_2&\leq 2^p\int_{(A_{k_0})^c}\big|\mu^{\rho}_{\Omega}(F_2)(x)\big|^p\,dx\\
&\leq 2^p\sum_{k=k_0+1}^\infty\sum_i
\int_{\big(\widetilde{Q^k_i}\big)^c}\big|\mu^{\rho}_{\Omega}\big(b^k_i\big)(x)\big|^p\,dx.
\end{split}
\end{equation*}
As before, for $1\leq\rho<n$, if we set
$$\mathcal K^{\rho}(x,z)=\frac{\Omega(x,z)}{|z|^{n-\rho}}\chi_{\{|z|\le1\}}(z) \quad \mbox{and} \quad \mathcal K^{\rho}_t(x,z)=\frac{1}{t^n}\cdot \mathcal K^{\rho}\Big(x,\frac{z}{t}\Big),$$
then
\begin{equation}
\mu^{\rho}_{\Omega}\big(b^k_i\big)(x)=\left(\int_0^\infty\bigg|\int_{\mathbb R^n}\mathcal K^{\rho}_t(x,x-y)b^k_i(y)\,dy\bigg|^2\frac{dt}{t}\right)^{1/2}.
\end{equation}
For any $x\in\big(\widetilde{Q^k_i}\big)^c$, supp\,$\mathcal K^{\rho}(x,\cdot)\subseteq B(0,1)$, the unit ball in $\mathbb R^n$. By the cancellation condition of $b^k_i\in L^\infty(\mathbb R^n)$, we can get
\begin{equation*}
\begin{split}
\bigg|\int_{\mathbb R^n}\mathcal K^{\rho}_t(x,x-y)b^k_i(y)\,dy\bigg|=\ &
\left|\int_{Q^k_i}\Big[\mathcal K^{\rho}_t(x,x-y)-\mathcal K^{\rho}_t(x,x-x^k_i)\Big]
b^k_i(y)\,dy\right|\\
\le\ &\frac{C}{t^{\rho}}\cdot\int_{Q^k_i}
\bigg|\frac{1}{|x-y|^{n-\rho}}-\frac{1}{|x-x^k_i|^{n-\rho}}\bigg|\big|b^k_i(y)\big|\,dy\\
&+\frac{1}{t^{\rho}}\cdot\int_{Q^k_i}\frac{|\Omega(x,x-y)-\Omega(x,x-x^k_i)|}{|x-x^k_i|^{n-\rho}}
\big|b^k_i(y)\big|\,dy\\
=\ &\mbox{\upshape I+II}.
\end{split}
\end{equation*}
If $y\in Q^k_i$ and $x\in\big(\widetilde{Q^k_i}\big)^c$, then we still have $|x-y|\sim|x-x^k_i|$ for all $i$ and $k$. Again, we apply the mean value theorem to obtain
\begin{equation*}
\begin{split}
\mbox{\upshape I}&\leq\frac{C}{t^{\rho}}\int_{Q^k_i}\frac{|y-x^k_i|}{|x-x^k_i|^{n-\rho+1}}\big|b^k_i(y)\big|\,dy\\
&\leq C\cdot\big\|b^k_i\big\|_{L^\infty}\cdot\frac{(r^k_i)^{n+1}}{t^{\rho}|x-x^k_i|^{n-\rho+1}}.
\end{split}
\end{equation*}
In addition, it follows from the previous inequalities (\ref{ineq-2}) and (\ref{Omega-2}) that
\begin{equation*}
\mbox{\upshape II}\le C\cdot\big\|b^k_i\big\|_{L^\infty}\cdot
\frac{(r^k_i)^{n+\sigma_1}}{t^{\rho}|x-x^k_i|^{n-\rho+\sigma_1}\big(\log\frac{|x-x^k_i|}{r^k_i}\big)^{\sigma_2}}.
\end{equation*}
Recall that for any fixed $x$, supp\,$\mathcal K^{\rho}(x,\cdot)\subseteq\{z\in\mathbb R^n:|z|\leq 1\}$. When $y\in Q^k_i$ and $x\in\big(\widetilde{Q^k_i}\big)^c$, then a direct calculation shows that
\begin{equation}
t\ge|x-y|\ge\big|x-x^k_i\big|-\big|y-x^k_i\big|\ge\frac{|x-x^k_i|}{2}.
\end{equation}
Summarizing the above two estimates for I and II, for any $x\in \big(\widetilde{Q^k_i}\big)^c$, we have
\begin{equation*}
\begin{split}
\big|\mu^{\rho}_\Omega\big(b^k_i\big)(x)\big|&\leq C\cdot\big\|b^k_i\big\|_{L^\infty}
\Bigg[\frac{(r^k_i)^{n+1}}{|x-x^k_i|^{n-\rho+1}}+
\frac{(r^k_i)^{n+\sigma_1}}{|x-x^k_i|^{n-\rho+\sigma_1}\big(\log\frac{|x-x^k_i|}{r^k_i}\big)^{\sigma_2}}\Bigg]\\
&\times\left(\int_{\frac{|x-x^k_i|}{2}}^\infty\frac{dt}{t^{2\rho+1}}\right)^{1/2}\\
&\leq C\cdot \big\|b^k_i\big\|_{L^\infty}
\Bigg[\frac{(r^k_i)^{n+1}}{|x-x^k_i|^{n+1}}+
\frac{(r^k_i)^{n+\sigma_1}}{|x-x^k_i|^{n+\sigma_1}\big(\log\frac{|x-x^k_i|}{r^k_i}\big)^{\sigma_2}}\Bigg].
\end{split}
\end{equation*}
Repeating the arguments used in the proof of Theorem \ref{mainthm:4}, we can also prove that
\begin{equation*}
K''_2\le C\big\|f\big\|^p_{WH^p}.
\end{equation*}
Summing up all the above estimates and then taking the supremum over all $\lambda>0$, we conclude the proof of Theorem \ref{mainthm:6}.
\end{proof}

At the end of this section, we remark that for any function $f$, a straightforward computation shows that the radial maximal function of $f$ is pointwise dominated by $M(f)$, where $M$ denotes the standard Hardy--Littlewood maximal operator. Hence, by the weak type (1,1) estimate of $M$, it is easy to see that the space $L^1(\mathbb R^n)$ is continuously embedded as a subspace of $WH^1(\mathbb R^n)$, and we have $\|f\|_{WH^1}\le C\|f\|_{L^1}$ for any $f\in L^1(\mathbb R^n)$. Therefore, as direct consequences of Theorems \ref{mainthm:4}--\ref{mainthm:6}, we immediately obtain the following result.

\newtheorem{corollary}[theorem]{Corollary}

\begin{corollary}\label{cor:1}
Let $n\geq 2$ and $\Omega(x,z)$ satisfy $(\ref{cancel})$ and the $L^{\sigma_1}$-$(\log L)^{\sigma_2}$ condition $(\ref{L-logL})$. Then $T_{\Omega}$ is bounded from $L^1(\mathbb R^n)$ into $WL^1(\mathbb R^n)$ provided that $\sigma_1$ and $\sigma_2$ satisfy either of the following\\
$(i)$ $\sigma_1=0$ and $\sigma_2>2$; \\
$(ii)$ $0<\sigma_1\leq1$ and $\sigma_2\geq0$.
\end{corollary}

\begin{corollary}\label{cor:2}
Let $0<\alpha<n$, $1/q=1-\alpha/n$ and $\Omega(x,z)$ satisfy the $L^{\sigma_1}$-$(\log L)^{\sigma_2}$ condition $(\ref{L-logL})$. Then $T_{\Omega,\alpha}$ is bounded from $L^1(\mathbb R^n)$ into $WL^q(\mathbb R^n)$ provided that $\sigma_1$ and $\sigma_2$ satisfy either of the following\\
$(i)$ $\sigma_1=0$ and $\sigma_2>1/q+1$;\\
$(ii)$ $0<\sigma_1\leq1$ and $\sigma_2\geq0$.
\end{corollary}

\begin{corollary}\label{cor:3}
Let $1\leq \rho<n$ and $\Omega(x,z)$ satisfy $(\ref{cancel})$ and the $L^{\sigma_1}$-$(\log L)^{\sigma_2}$ condition $(\ref{L-logL})$. Then $\mu^{\rho}_{\Omega}$ is bounded from $L^1(\mathbb R^n)$ into $WL^1(\mathbb R^n)$ provided that $\sigma_1$ and $\sigma_2$ satisfy either of the following\\
$(i)$ $\sigma_1=0$ and $\sigma_2>2$; \\
$(ii)$ $0<\sigma_1\leq1$ and $\sigma_2\geq0$.
\end{corollary}

It is worth pointing out that the conclusions $(i)$ of Corollaries \ref{cor:1}--\ref{cor:3} were also given by the author in \cite{wang}.

\section{Boundedness on the Hardy--Lorentz spaces $H^{p,q}(\mathbb R^n)$}

For $0<p<\infty$ and $0<q\leq\infty$, the Lorentz space $L^{p,q}(\mathbb R^n)$ consists of those measurable functions $f$ with finite quasi-norm $\|f\|_{p,q}$ given by
\begin{equation*}
\|f\|_{L^{p,q}}=
\begin{cases}\displaystyle\left(\frac{\,q\,}{p}\int_0^\infty [t^{1/p}f^*(t)]^q\frac{dt}{t}\right)^{1/q}, \quad & 0<q<\infty,\\
\displaystyle\sup_{t>0}\,[t^{1/p}f^*(t)], & q=\infty.
\end{cases}
\end{equation*}
where $f^*$ is the nonincreasing rearrangement of $f$ on $(0,\infty)$. Note that, in particular, $L^{p,p}(\mathbb R^n)=L^{p}(\mathbb R^n)$ and $L^{p,\infty}(\mathbb R^n)=WL^{p}(\mathbb R^n)$. For any $0<p<\infty$ and $p<q<\infty$, from the theory of real interpolation, we have the following result (see \cite{bergh,grafakos})
\begin{equation}\label{Lorentz}
\big(L^{p,p},L^{p,\infty}\big)_{\theta,q}=L^{p,q},
\end{equation}
where $1/q={(1-\theta)}/p+\theta/{\infty}={(1-\theta)}/p$ and $0<\theta<1$.

Just as in the case of $H^p(\mathbb R^n)$, the Hardy--Lorentz spaces $H^{p,q}(\mathbb R^n)$ can also be defined in terms of radial maximal functions for all $0<p\leq1$ and $0<q\leq\infty$. Let $\varphi$ be a function in $\mathscr S(\mathbb R^n)$ satisfying $\int_{\mathbb R^n}\varphi(x)\,dx=1$ and define the radial maximal function $M_\varphi(f)=\sup_{t>0}\big|(\varphi_t*f)(x)\big|$. Then the Hardy--Lorentz space $H^{p,q}(\mathbb R^n)$ consists of those tempered distributions $f\in \mathscr S'(\mathbb R^n)$ for which
$M_\varphi(f)\in L^{p,q}(\mathbb R^n)$ with $\big\|f\big\|_{H^{p,q}}=\big\|M_\varphi(f)\big\|_{L^{p,q}}$. Moreover, for any $0<p\leq1$, we can see that $H^{p,p}(\mathbb R^n)=H^{p}(\mathbb R^n)$ and $H^{p,\infty}(\mathbb R^n)=WH^{p}(\mathbb R^n)$. For more information about the properties and applications of Hardy--Lorentz spaces, the reader is referred to \cite{cfefferman,abu}. For all $0<p\leq1$ and $p<q<\infty$, we need the following interpolation result for the Hardy--Lorentz spaces $H^{p,q}(\mathbb R^n)$ with the same first index $p$, which was shown by Abu-Shammala and Torchinsky in \cite{abu}.
\begin{equation}\label{Hardy-Lorentz}
\big(H^{p,p},H^{p,\infty}\big)_{\theta,q}=H^{p,q},
\end{equation}
where $1/q={(1-\theta)}/p+\theta/{\infty}={(1-\theta)}/p$ and $0<\theta<1$. Therefore, using the facts (\ref{Lorentz}) and (\ref{Hardy-Lorentz}) mentioned above, together with the main theorems stated in Section 1, we finally obtain

\begin{theorem}
Let $n\geq 2$, $n/{(n+1)}<p\leq1$, and $\Omega(x,z)$ satisfy $(\ref{cancel})$ and the $L^{\sigma_1}$-$(\log L)^{\sigma_2}$ condition $(\ref{L-logL})$. Then for any $p<q<\infty$, $T_{\Omega}$ is bounded from $H^{p,q}(\mathbb R^n)$ into $L^{p,q}(\mathbb R^n)$ provided that $\sigma_1$ and $\sigma_2$ satisfy either of the following\\
$(i)$ $\sigma_1=n(1/p-1)$ and $\sigma_2>2/p$; \\
$(ii)$ $n(1/p-1)<\sigma_1\leq1$ and $\sigma_2\geq0$.
\end{theorem}

\begin{theorem}
Let $0<\alpha<n$, $n/{(n+1)}<p\leq1$, $1/q=1/p-\alpha/n$ and $\Omega(x,z)$ satisfy the $L^{\sigma_1}$-$(\log L)^{\sigma_2}$ condition $(\ref{L-logL})$. Then for any $q<s<\infty$, $T_{\Omega,\alpha}$ is bounded from $H^{p,s}(\mathbb R^n)$ into $L^{q,s}(\mathbb R^n)$ provided that $\sigma_1$ and $\sigma_2$ satisfy either of the following\\
$(i)$ $\sigma_1=n(1/q-1)+\alpha$ and $\sigma_2>1/q+\max\{1,1/q\}$;\\
$(ii)$ $n(1/q-1)+\alpha<\sigma_1\leq1$ and $\sigma_2\geq0$.
\end{theorem}

\begin{theorem}
Let $1\leq \rho<n$, $n/{(n+1)}<p\leq1$, and $\Omega(x,z)$ satisfy $(\ref{cancel})$ and the $L^{\sigma_1}$-$(\log L)^{\sigma_2}$ condition $(\ref{L-logL})$. Then for any $p<q<\infty$, $\mu^{\rho}_{\Omega}$ is bounded from $H^{p,q}(\mathbb R^n)$ into $L^{p,q}(\mathbb R^n)$ provided that $\sigma_1$ and $\sigma_2$ satisfy either of the following\\
$(i)$ $\sigma_1=n(1/p-1)$ and $\sigma_2>2/p$; \\
$(ii)$ $n(1/p-1)<\sigma_1\leq1$ and $\sigma_2\geq0$.
\end{theorem}

\end{document}